\documentclass[10pt]{article}

\usepackage{amsfonts}
\usepackage{amssymb}
\usepackage{amsthm}
\usepackage{mathrsfs}
\usepackage{amsmath}
\usepackage{bm}
\usepackage{amscd} 

\usepackage{tikz-cd}
\usepackage[all]{xy}

\usepackage{tikz}
\usetikzlibrary{matrix}

\usepackage{hyperref}

\numberwithin{equation}{section} \DeclareMathSizes{2}{10}{12}{13}
\newtheorem{thm}{Proposition}[section]
\newtheorem{Thm}[thm]{Theorem}

\newtheorem{cor}[thm]{Corollary}
\newtheorem{lem}[thm]{Lemma}
\newtheorem{defn}[thm]{Definition}

\newtheorem{ithm}{Theorem} 

\title{$FI$-modules over preadditive categories and torsion theories}

\author{Abhishek Banerjee\footnote{Partially supported by SERB Matrics fellowship MTR/2017/000112} }

\date{}

\begin{document}

\maketitle

\medskip
\begin{center}
\emph{Dept. of Mathematics, Indian Institute of Science, Bangalore - 560012, India.}

\emph{Email : abhishekbanerjee1313@gmail.com}
\end{center}

\medskip

\begin{abstract}
We work with $FI$-modules over a small preadditive category $\mathcal R$, viewed as a ring with several objects. Our aim is to study torsion theories for $FI$-modules. We are especially interested in torsion theories on finitely generated $FI$-modules and the category of 
what we call ``shift finitely generated'' $FI$-modules. We also apply these methods to study inductive descriptions of $FI$-modules over $\mathcal R$. 
\end{abstract}

\medskip
MSC(2010) Subject Classification:  18E40

\medskip
Keywords : $FI$-modules, torsion theories

\section{Introduction}

\smallskip
Let $FI$ be the category of finite sets and injective maps. If $R$ is a ring, an $FI$-module over $R$  is a functor from $FI$ to the category of $R$-modules. The notion of $FI$-modules was introduced by Church, Ellenberg and Farb in \cite{CEFx}, with a view towards a deeper understanding of the Church-Farb theory \cite{CF} of representation stability for $S_n$-representations. This was further developed by Church, Ellenberg and Farb \cite{CEF-1}, \cite{CEF-2}, by Church, Ellenberg, Farb and Nagpal \cite{Four1}, by Church and Ellenberg \cite{CEx}, by Putman \cite{Put}, by Putman and Sam \cite{PS}, and by Sam and Snowden \cite{SS0} \cite{SS}.  Since then, a wide variety of results on $FI$-modules has been developed by numerous authors, with applications to algebraic topology, algebraic geometry and representation theory (see, for instance, \cite{CMNR}, \cite{LE}, \cite{LGO}, \cite{RSS}, \cite{Ramos}). 

\smallskip
In this paper, our aim is to study torsion theories for $FI$-modules. We are especially interested in torsion theories on finitely generated $FI$-modules and the category of 
what we call ``shift finitely generated'' $FI$-modules (see Definition \ref{D3.9k}). 
 We work with $FI$-modules over a small preadditive category $\mathcal R$, viewed as a ring with several objects
in the sense of Mitchell \cite{Mi}. As such, the category $FI_{\mathcal R}$ of  $FI$-modules over $\mathcal R$ consists of functors from $FI$ to the category $Mod-\mathcal R$ of right modules over $\mathcal R$. We recall that a right module over $\mathcal R$ is a functor from $\mathcal R^{op}$ to the category of abelian groups. We mostly work with the case where $\mathcal R$ is such that the category $Mod-\mathcal R$ is locally noetherian. When $\mathcal R$
is abelian, we have studied in \cite{ABCCM} how torsion theories on $\mathcal R$ may be extended to certain classes of modules over $\mathcal R$. 

\smallskip
We begin with a hereditary torsion theory $(\mathcal T,\mathcal F)$ on $Mod-\mathcal R$. We show that $\tau$ induces a hereditary torsion class $\overline{\mathcal T}$ on the subcategory
$FI^{fg}_{\mathcal R}$ of finitely generated $FI$-modules as well as a hereditary torsion class $\overline{\mathcal T}^{sfg}$ on the category
$FI^{sfg}_{\mathcal R}$ of shift finitely generated $FI$-modules. We then extend $\overline{\mathcal T}$ and $\overline{\mathcal T}^{sfg}$ to Serre subcategories $\hat{\mathcal T}$
and $\widetilde{\mathcal T}$ respectively of $FI_{\mathcal R}$.
 In other words, we have $\hat{\mathcal T}\cap FI^{fg}_{\mathcal R}=\overline{\mathcal T}$ and $\widetilde{\mathcal T}\cap FI^{sfg}_{\mathcal R}=\overline{\mathcal T}^{sfg}$. We then describe functors from $FI_{\mathcal R}$ to $\hat{\mathcal T}$-closed 
and $\widetilde{\mathcal T}$-closed objects of $FI_{\mathcal R}$. In fact, we show that an object in $FI_{\mathcal R}$ is closed with respect to the Serre subcategory $\hat{\mathcal T}$ if and only if 
it is closed with respect to  $\widetilde{\mathcal T}$. Finally, we apply these methods to study inductive descriptions of $FI$-modules over $\mathcal R$. 

\smallskip
We begin in Section 2 with preliminary results on $FI$-modules over $\mathcal R$, extending those from \cite[$\S$ 2.1]{Four1}. In particular, we show that $FI_{\mathcal R}$ is a Grothendieck category with a set of finitely generated projective generators. We also recall that when $\mathcal R$ is such that $Mod-\mathcal R$ is locally noetherian, the category 
$FI_{\mathcal R}$ of $FI$-modules over $\mathcal R$ becomes a locally noetherian category (see  (see \cite[Theorem A]{Four1}, \cite{Djament},  \cite{PS}, \cite{SS} and  \cite[Theorem 9.1]{Krause1})). For each $a\geq 0$, the category $FI_{\mathcal R}$ is equipped with a shift functor defined by setting 
\begin{equation*}
\mathbb S^a:FI_{\mathcal R}\longrightarrow FI_{\mathcal R} \qquad \mathbb S^a\mathscr V(S):=\mathscr V(S\sqcup [-a])
\end{equation*} for each $\mathscr V\in FI_{\mathcal R}$ and each finite set $S$, where $[-a]$ is a fixed set of cardinality $a$. In \cite[$\S$ 2.3]{Four1}, an $FI$-module $V$ over a ring
$R$ is said to be torsion if it satisfies
\begin{equation}
V=\underset{a\geq 0}{\bigcup} \textrm{ }Ker(X^a:V\mapsto S^aV)
\end{equation} Here, $S^a$ is the $a$-th shift functor on $FI$-modules over $R$  and $X^a$ is the canonical morphism $V\longrightarrow S^aV$ induced by the inclusion $T\hookrightarrow T
\sqcup [-a]$ for each finite set $T$. In Section 3, we consider a torsion theory $\tau=(\mathcal T,\mathcal F)$ on $Mod-\mathcal R$ and the subcategory of finitely generated $FI$-modules determined by
setting
\begin{equation}\label{intrtrskc}
Ob(\overline{\mathcal T}):=\{\mbox{$\mathscr V\in Ob(FI_{\mathcal R}^{fg})$ $\vert$ $\mathscr V_n\in \mathcal T$ for $n\gg 0$}\}
\end{equation}
Our first result describes the induced torsion theory on $FI^{fg}_{\mathcal R}$ and a formula for the torsion subobject of a finitely generated $FI$-module. 

\begin{ithm} (see \ref{xxP3.2} and \ref{P3.8cf})  Let $\mathcal R$ be a small preadditive category such that
$Mod-\mathcal R$ is locally noetherian. Let $(\mathcal T,\mathcal F)$ be a torsion theory on $Mod-\mathcal R$.

\smallskip
(a) Then, $\overline{\mathcal T}$ is a torsion 
class in the category $FI_{\mathcal R}^{fg}$ of finitely generated $FI$-modules over $\mathcal R$. Additionally, if $\mathcal T$ is a hereditary torsion
class, so is $\overline{\mathcal T}$.  

\smallskip
(b) Suppose that $\mathcal T$ is hereditary. Then, for any $\mathscr V\in FI_{\mathcal R}^{fg}$, the torsion subobject of $\mathscr V$ with respect to the torsion
class $\overline{\mathcal T}$ is given by $\overline{\mathcal T}(\mathscr V)$, where  
\begin{equation}\label{intrtors3}
\overline{\mathcal T}(\mathscr V)(S):=\underset{a\geq 0}{colim}\textrm{ }lim\left(
\begin{CD}\mathscr V(S)@>>>({\mathbb S}^a\mathscr V)(S)@<<<\mathcal T( ({\mathbb S}^a\mathscr V)(S))
\end{CD}\right)
\end{equation} for each finite set $S$. 
\end{ithm}

\smallskip
We will say that an $FI$-module $\mathscr V$ is shift finitely generated if there exists $d\geq 0$ such that $\mathbb S^d\mathscr V$ is finitely generated. We denote by
$FI^{sfg}_{\mathcal R}$ the full subcategory of shift finitely generated $FI$-modules. Given the torsion
theory $(\mathcal T,\mathcal F)$ on $Mod-\mathcal R$, we now consider
\begin{equation}\label{intrtrskc2}
Ob(\overline{\mathcal T}^{sfg}):=\{\mbox{$\mathscr V\in Ob(FI_{\mathcal R}^{sfg})$ $\vert$ Every finitely generated $\mathscr W\subseteq \mathscr V$ lies in $\overline{\mathcal T}$}\}
\end{equation} The next result shows that $FI^{sfg}_{\mathcal R}$ is a Serre subcategory of $FI_{\mathcal R}$ and that the formula in \eqref{intrtors3}
may be extended to   describe the induced torsion theory on $FI^{sfg}_{\mathcal R}$. For this, we also obtain some intermediate results on torsion in locally noetherian Grothendieck
categories.

\begin{ithm} (see \ref{P3.11}, \ref{P3.13} and \ref{P3.17cf}) Suppose that $Mod-\mathcal R$ is locally noetherian. 

\smallskip
(a) Then, the full  subcategory $FI^{sfg}_{\mathcal R}$ given by 
\begin{equation}
Ob(FI^{sfg}_{\mathcal R}):=\{\mbox{$\mathscr V\in Ob(FI_{\mathcal R})$ $\vert$ $\mathbb S^d\mathscr V$ is finitely generated for some $d\geq 0$ }\}
\end{equation}
is a Serre subcategory of $FI_{\mathcal R}$, i.e., it is closed under extensions, quotients and subobjects.

\smallskip
(b) If $\tau=(\mathcal T,\mathcal F)$ is a hereditary torsion theory on $Mod-\mathcal R$, then
$\overline{\mathcal T}^{sfg}$ is a hereditary torsion class in $FI^{sfg}_{\mathcal R}$.  For any $\mathscr V\in FI_{\mathcal R}^{sfg}$, the torsion subobject of $\mathscr V$ with respect to the torsion
class $\overline{\mathcal T}^{sfg}$ is given by 
\begin{equation}\label{intrtors317}
\overline{\mathcal T}^{sfg}(\mathscr V)(S):=\underset{a\geq 0}{colim}\textrm{ }lim\left(
\begin{CD}\mathscr V(S)@>>>({\mathbb S}^a\mathscr V)(S)@<<<\mathcal T( ({\mathbb S}^a\mathscr V)(S))
\end{CD}\right)
\end{equation} for each finite set $S$.

\end{ithm}

In the rest of this paper, we always suppose that $\tau=(\mathcal T,\mathcal F)$ is a hereditary torsion  theory on $Mod-\mathcal R$. In Section 4, we consider the subcategory
$\hat{\mathcal T}$ of $FI_{\mathcal R}$ defined by setting
\begin{equation}\label{intr4.3ne}
\begin{array}{c}
Ob({\hat{\mathcal T}}):=\{\mbox{$\mathscr V\in Ob(FI_{\mathcal R})$ $\vert$ $\mathscr V_n\in \mathcal T$ for $n\gg 0$}\}\\
\end{array}
\end{equation} It is clear that we have $\hat{\mathcal T}\cap FI^{fg}_{\mathcal R}=\overline{\mathcal T}$. While $\hat{\mathcal T}$ is not a torsion class (it is not necessarily
closed under direct sums), we observe that it is a Serre subcategory of $FI_{\mathcal R}$.  

\smallskip
We recall that an object $L$ in a Grothendieck category $\mathcal A$ is said to be closed with respect to a Serre subcategory $\mathcal C\subseteq \mathcal A$ if $Hom(u,L):
Hom(B,L)\longrightarrow Hom(A,L)$ is an isomorphism for every $u:A\longrightarrow B$ in $\mathcal A$ such that $Ker(u)$, $Coker(u)\in \mathcal C$. We first develop a general result 
(see Proposition \ref{P4.4}) that in a locally noetherian Grothendieck category, it suffices to check this criterion with $A$, $B$ finitely generated. 

\smallskip Corresponding to each
$\mathscr V\in FI_{\mathcal R}$, we want to construct an object that is closed with respect to the Serre subcategory $\hat{\mathcal T}$. In other words, we will describe
a functor from $FI_{\mathcal R}$ taking values in the subcategory $Cl(\hat{\mathcal T})$ of $\hat{\mathcal T}$-closed objects.  For this, we first express $\hat{\mathcal T}$ as a union
$ {\hat{\mathcal T}}=\underset{a\geq 0}{\bigcup}
{\hat{\mathcal T^a}}$ where 
\begin{equation}
Ob({\hat{\mathcal T^a}}):=\{\mbox{$\mathscr V\in Ob(FI_{\mathcal R})$ $\vert$ $\mathscr V_n\in \mathcal T$ for all $n\geq a$}\} \qquad \forall\textrm{ }a\geq 0
\end{equation} We will also need the functor $\mathbb E_\tau$ which takes an $FI$-module $\mathscr V:FI\longrightarrow Mod-\mathcal R$ to its composition with 
the torsion envelope in $Mod-\mathcal R$. We obtain the following result.

\begin{ithm}\label{InTh3q} (see \ref{P4.10} and \ref{P4.11}) Let $Mod-\mathcal R$ be locally noetherian. Let $\tau=(\mathcal T,\mathcal F)$ be a hereditary torsion theory on $Mod-\mathcal R$. Then, we have a functor $\mathbb L_\tau:FI_{\mathcal R}\longrightarrow Cl(\hat{\mathcal T})$ and a canonical morphism
\begin{equation}
l_\tau (\mathscr V): \mathscr V\longrightarrow \mathbb L_\tau(\mathscr V):=\underset{k\geq 0}{\varinjlim}\textrm{ }\mathbb L_\tau^k(\mathscr V)
\end{equation} for each $\mathscr V\in FI_{\mathcal R}$. Here, $\mathbb L_\tau^k=\mathbb T^k\circ \mathbb S^k\circ \mathbb E_\tau$, where $\mathbb T^k:FI_{\mathcal R}
\longrightarrow FI_{\mathcal R}$ is the right adjoint to the shift functor $\mathbb S^k:FI_{\mathcal R}
\longrightarrow FI_{\mathcal R}$.

\end{ithm}

In Section 5, our aim is to prove a result similar to Theorem \ref{InTh3q} for shift finitely generated objects by considering the subcategory
\begin{equation}\label{intreq5.1nw}
Ob(\widetilde{\mathcal T}):=\{\mbox{$\mathscr V\in Ob(FI_{\mathcal R})$ $\vert$ Every finitely generated $\mathscr W\subseteq \mathscr V$ lies in $\overline{\mathcal T}$
 }\}
\end{equation}
which satisfies  $\widetilde{\mathcal T}\cap FI^{sfg}_{\mathcal R}=\overline{\mathcal T}^{sfg}$. Unlike in the case of $\hat{\mathcal T}$ which is only a Serre subcategory, we will show that $\widetilde{\mathcal T}$
is actually a hereditary torsion class. Further, the $\widetilde{\mathcal T}$-closed objects actually coincide with the $\hat{\mathcal T}$-closed objects. In other words, we have the following result.

\begin{ithm} (see \ref{P5.4nw} and \ref{Tmh5.8}) Let  $Mod-\mathcal R$ be a locally noetherian category and $\tau=(\mathcal T,\mathcal F)$ a hereditary torsion theory on $Mod-\mathcal R$. Then,

\smallskip
(a) The full subcategory $\widetilde{\mathcal T}$ is a hereditary torsion class.

\smallskip
(b) An object in $FI_{\mathcal R}$ is closed with respect to $\widetilde{\mathcal T}$ if and only if it is closed with respect to $\hat{\mathcal T}$, i.e., 
$Cl(\hat{\mathcal T})=Cl(\widetilde{\mathcal T})$. In particular, we have a functor $\mathbb L_\tau:FI_{\mathcal R}\longrightarrow Cl(\hat{\mathcal T})=Cl(\widetilde{\mathcal T})$. 

\end{ithm}

In Section 6, we begin by providing an inductive description for finitely generated $FI$-modules over $\mathcal R$. For this, we have to work with functors
$H_a:FI_{\mathcal R}\longrightarrow FI_{\mathcal R}$, $a\geq 0$,  which are defined as homology groups of a complex similar to the construction in  \cite[$\S$ 2.4]{Four1}. The following
result is analogous to \cite[Theorem C]{Four1}.

\begin{ithm}\label{InTh5q} (see \ref{T5.5}) Suppose that $Mod-\mathcal R$ is locally noetherian. Let $\mathscr V\in FI_{\mathcal R}$ be a finitely generated object.  Then, there exists $N\geq 0$ such that
\begin{equation}\label{intr5.17xe}
\underset{\tiny \begin{array}{c} T\subseteq S \\ |T|\leq N\\ \end{array}}{colim}\textrm{ }\mathscr V(T)=\mathscr V(S)
\end{equation} for each finite set $S$.

\end{ithm}

We conclude by proving a result similar to Theorem \ref{InTh5q} for shift finitely generated objects in $FI_{\mathcal R}$. For this, we apply the methods developed in previous
sections to the zero torsion class on $Mod-\mathcal R$. Our final result is as follows.

\begin{ithm}  (see \ref{P5.8} and \ref{Cor6.9}) Suppose that $Mod-\mathcal R$ is locally noetherian. Let $\mathscr V\in FI_{\mathcal R}$ be a shift finitely generated object. 

\smallskip
(a) Fix $a\geq 0$ and consider  any finitely generated subobject $\mathscr W\subseteq H_a(\mathscr V)$. Then,  there exists
$N\geq 0$ such that $\mathscr W_n=0$ for all $n\geq N$. 

\smallskip
(b)  Let $\mathscr V\in FI_{\mathcal R}$ be a shift finitely generated object such that
$H_0(\mathscr V)$ and $H_1(\mathscr  V)$ are finitely generated. Then, there exists $N\geq 0$ such that
\begin{equation}\label{intr5.17xet}
\underset{\tiny \begin{array}{c} T\subseteq S \\ |T|\leq N\\ \end{array}}{colim}\textrm{ }\mathscr V(T)=\mathscr V(S)
\end{equation} for each finite set $S$.

\end{ithm}

\section{Finitely generated $FI$-modules over rings with several objects}

\smallskip

Let $FI$ denote the category consisting of finite sets and injections. For each $n> 0$, we set $[n]:=\{1,2,...,n\}$   while $[0]$ is taken to be the empty set. 
Then, $FI$ is equivalent to its full subcategory consisting of the objects $[n]$ for $n\geq 0$. In particular, while $FI$ is not a small category, we see that it is essentially small, i.e., 
equivalent to a small category. 

\smallskip
We now let $\mathcal R$ be a small preadditive category, viewed as a ring with several objects. Then, a (right) $\mathcal R$-module is a functor $\mathcal R^{op}\longrightarrow
Ab$, where $Ab$ is the category of abelian groups. The category of right $\mathcal R$-modules will be denoted by $Mod-\mathcal R$. For each object $r\in \mathcal R$, we
set $H_r:=\mathcal R(\_\_,r):\mathcal R^{op}\longrightarrow
Ab$. 

\smallskip
It is well known (see, for instance, \cite[$\S$ 1.4]{Gar}) that $Mod-\mathcal R$ is a locally finitely presented Grothendieck category, with the collection $\{H_r\}_{r\in \mathcal R}$ being a family of finitely generated
projective generators. We notice that
since $Mod-\mathcal R$ is a Grothendieck category, it is well-powered (see, for instance, \cite[Proposition IV.6.6]{Bo}), i.e., the collection of (equivalence classes of) subobjects of any $V\in Mod-\mathcal R$ is a set. 

\begin{defn}\label{D2.1} Let $\mathcal R$ be a small preadditive category. An $FI$-module over $\mathcal R$ is a functor from $FI$ to $Mod-\mathcal R$. For any such $\mathscr V:
FI\longrightarrow Mod-\mathcal R$, we set $\mathscr V_n:=\mathscr V(n)=\mathscr V([n])$ for each $n\geq 0$. For any morphism $\phi: S\longrightarrow T$ in $FI$, we denote by $\phi_*:\mathscr V(S)\longrightarrow \mathscr V(T)$ the induced morphism $\mathscr V(\phi)$. 

\smallskip
The category of $FI$-modules over $\mathcal R$ will be denoted by $FI_{\mathcal R}$. 

\end{defn}

 Since $FI$ is  essentially small, it follows from \cite[Theorem 14.2]{Faith} that the category of $FI$-modules over $\mathcal R$ is a Grothendieck category. In particular, for any morphism
$\tilde f:\mathscr V\longrightarrow \mathscr V'$ in $FI_{\mathcal R}$, we have
\begin{equation}
Ker(\tilde f)(S)=Ker(\tilde f(S):\mathscr V(S)\longrightarrow\mathscr V'(S))\qquad Coker(\tilde f)(S)=Coker(\tilde f(S):\mathscr V(S)\longrightarrow\mathscr V'(S))
\end{equation} for any $S\in FI$. In this paper, unless otherwise mentioned, by an $FI$-module, we will always mean an $FI$-module over
$\mathcal R$. For any $\mathscr V\in FI_{\mathcal R}$, we set
\begin{equation}
el(\mathscr V):=\underset{d\geq 0}{\coprod}\textrm{ } \underset{r\in \mathcal R }{\coprod}\mathscr V(d)(r)
\end{equation}

\begin{defn}\label{D2.2}
Fix $d\geq 0$. An $FI$-module $\mathscr V$ is said to be  generated in degree $\leq d$  if there exists a  (not necessarily finite) collection $\{f_i\vert i\in I  \}\subseteq \underset{e\leq d}{\coprod}\textrm{ } \underset{r\in \mathcal R }{\coprod}\mathscr V(e)(r)$ having the property
that any subobject 
 $\mathscr V'\subseteq \mathscr V$ such that $\{f_i\vert i\in I  \}\subseteq el(\mathscr V')$ must satisfy $\mathscr V'=\mathscr V$.

\end{defn}

\smallskip
Given  finite sets $S$, $T\in FI$, we will denote by $(S,T)$ the set of injections $S\hookrightarrow T$, i.e., the morphisms from $S$ to $T$ in the category $FI$. 
For any $r\in\mathcal R$ and $d\geq 0$, we now define $_d\mathscr M_r\in FI_{\mathcal R}$ as follows:
\begin{equation}\label{mdr}
_d\mathscr M_r:FI\longrightarrow Mod-\mathcal R\qquad S\mapsto H_r^{([d],S)}
\end{equation} where $H_r^{([d],S)}$ denotes the direct sum of copies of $H_r$ indexed by the set $([d],S)$. 

\begin{lem}\label{L2.3} Let $\mathscr V\in FI_{\mathcal R}$.
For any $d\geq 0$ and any $r\in \mathcal R$, we have a canonical isomorphism 
\begin{equation}
{FI_{\mathcal R}}(_d\mathscr M_r,\mathscr V)\cong \mathscr V(d)(r)
\end{equation} of abelian groups.
\end{lem}

\begin{proof}
 By Yoneda Lemma, an element $f\in \mathscr V(d)(r)$ corresponds to a morphism $f:H_r\longrightarrow \mathscr V(d)$ in $Mod-\mathcal R$. For any finite set $S$, we can take a direct sum of copies of $f$ to obtain a morphism $f^{([d],S)}: H_r^{([d],S)}\longrightarrow \mathscr V(d)^{([d],S)}$ in $Mod-\mathcal R$. Since $\mathscr V$ is a covariant functor from
 $FI$ to $Mod-\mathcal R$,  each morphism $\phi\in ([d],S)$ induces a morphism $\mathscr V(\phi):\mathscr V(d)\longrightarrow \mathscr V(S)$. Together, these 
 determine a morphism $\mathscr V(d)^{([d],S)}\longrightarrow \mathscr V(S)$ from the  direct sum $\mathscr V(d)^{([d],S)}$. Composing with $f^{([d],S)}: H_r^{([d],S)}\longrightarrow \mathscr V(d)^{([d],S)}$, we obtain
 $\tilde{f}(S): {_d\mathscr M_r}(S)=H_r^{([d],S)}\longrightarrow \mathscr V(S)$ in $Mod-\mathcal R$. Since these morphisms are functorial with respect to $S\in FI$, the element $f\in \mathscr V(d)(r)$  
 determines a morphism $\tilde{f}:{_d\mathscr M_r}\longrightarrow \mathscr V$ in $FI_{\mathcal R}$. 
 
 \smallskip
 Conversely, suppose that we are given a morphism $\tilde{f}: {_d\mathscr M_r}\longrightarrow \mathscr V$ in $FI_{\mathcal R}$. In particular, this gives us a morphism
 $\tilde{f}(d):{_d\mathscr M_r}(d)=H_r^{([d],[d])}\longrightarrow \mathscr V(d)$ in $Mod-\mathcal R$. Considering the identity morphism $1_d\in ([d],[d])$ gives us an inclusion
 $H_r\longrightarrow H_r^{([d],[d])}$ which when composed with $\tilde{f}(d)$ gives a morphism $f:H_r\longrightarrow \mathscr V(d)$ in $Mod-\mathcal R$, i.e., an element $f\in \mathscr V(d)(r)$. It may be easily verified that these two associations are inverse to each other, which proves the result. 
\end{proof}

\begin{thm}\label{L2.4}
(a) For $d\geq 0$ and $r\in \mathcal R$, the object $_d\mathscr M_r$ is a finitely generated object of $FI_{\mathcal R}$.

\smallskip
(b) The collection $\{_d\mathscr M_r\}_{r\in \mathcal R,d\geq 0}$ is a set of generators for the Grothendieck category $FI_{\mathcal R}$. 

\smallskip
(c) An object $\mathscr V$
in $FI_{\mathcal R}$ is finitely generated if and only if there is an epimorphism
\begin{equation}
\underset{i\in I}{\bigoplus}\textrm{ } {_{d_i}\mathscr M_{r_i}}\longrightarrow \mathscr V
\end{equation} for some finite collection $\{(d_i,r_i)\}_{i\in I}$ with each $d_i\geq 0$ and $r_i\in\mathcal R$. 

\smallskip
(d) An object $\mathscr V$
in $FI_{\mathcal R}$ is finitely generated if and only if there is a finite collection $\{f_1,...,f_k\}\subseteq el(\mathscr V)$
such that any subobject $\mathscr V'\hookrightarrow \mathscr V$ with $\{f_1,...,f_k\}\subseteq el(\mathscr V')$ must satisfy $\mathscr V'=\mathscr V$. 

\smallskip
(e)  An object $\mathscr V$ in $FI_{\mathcal R}$ is generated in degree $\leq d$ if and only if there is an epimorphism
\begin{equation}
\underset{i\in I}{\bigoplus}\textrm{ } {_{d_i}\mathscr M_{r_i}}\longrightarrow \mathscr V
\end{equation} for some  collection $\{(d_i,r_i)\}_{i\in I}$ with each $0\leq d_i\leq d$ and $r_i\in\mathcal R$. 
\end{thm}

\begin{proof} We consider a filtered system $\{\mathscr W_j\}_{j\in J}$ in $FI_{\mathcal R}$ connected by monomorphisms and set $\mathscr W:=\underset{j\in J}{\varinjlim}
\textrm{ }\mathscr W_j$. By Lemma \ref{L2.3}, any morphism $\tilde f:{_d\mathscr M_r}\longrightarrow \mathscr W$ corresponds to an element $f\in \mathscr W(d)(r)$. Since
$\mathscr W(d)(r)=\underset{j\in J}{\varinjlim}
\textrm{ }\mathscr W_j(d)(r)$, we see that $\tilde f$ must factor through $\mathscr W_{j_0}$ for some $j_0\in J$. This proves (a). 

\smallskip
We have noted before that $FI_{\mathcal R}$ is a Grothendieck category. We consider some $\mathscr V\in FI_{\mathcal R}$ and some proper subobject $\mathscr V'\subsetneq \mathscr V$. Since the full subcategory of objects
$[n]$, $n\geq 0$ forms a skeleton of $FI$, we must have some $d\geq 0$ such that $\mathscr V'(d)\subsetneq \mathscr V(d)$ and therefore some $r\in \mathcal R$
such that $\mathscr V'(d)(r)\subsetneq \mathscr V(d)(r)$. Since ${FI_{\mathcal R}}(_d\mathscr M_r,\mathscr V)\cong \mathscr V(d)(r)$ by Lemma \ref{L2.3},
it follows that there exists a morphism $_d\mathscr M_r \longrightarrow \mathscr V$  in $FI_{\mathcal R}$ which does not factor through $\mathscr V'$. It follows from 
\cite[$\S$ 1.9]{Tohoku} that $\{_d\mathscr M_r\}_{r\in \mathcal R,d\geq 0}$ is a set of generators for $FI_{\mathcal R}$. This proves (b). The ``only if part'' of (c)  is clear from 
\cite[Proposition 1.9.1]{Tohoku}. The ``if part'' follows from the fact that a quotient of a finitely generated object is always finitely generated. Parts (d) and (e) also
follow easily by using Lemma \ref{L2.3}. 

\end{proof}

Similar to \cite[Definition 2.4]{Four1}, we now consider the following functor: for $\mathscr V\in FI_{\mathcal R}$, we define $H_0(\mathscr V):FI\longrightarrow Mod-\mathcal R$
by setting
\begin{equation}\label{h0}
\begin{array}{ll}
H_0(\mathscr V)(S)&:= Coker\left(\bigoplus \mathscr V(\phi): \underset{
\mbox{\tiny $\begin{array}{c}\phi:T\hookrightarrow S \\  |T|<|S| \\
\end{array}$}}{\bigoplus}
\mathscr V(T)\longrightarrow\mathscr V(S)\right) \\ &\textrm{ }= \mathscr V(S)/\left(\underset{
\mbox{\tiny $\begin{array}{c}\phi:T\hookrightarrow S \\  |T|<|S| \\
\end{array}$}}{\sum}Im(\mathscr V(\phi):\mathscr V(T)\longrightarrow \mathscr V(S))\right)\\
\end{array}
\end{equation}
From \eqref{h0}, it is clear that there is a canonical epimorphism $\mathscr V\longrightarrow H_0(\mathscr V)$ and that for any $\phi:T\longrightarrow S$ in $FI$
with $|T|<|S|$, we have $H_0(\mathscr V)(\phi)=0$. It follows that the functor $H_0$ is idempotent, i.e., $H_0^2=H_0$. 

\begin{lem}\label{Lem2.5} (a) The functor $H_0:FI_{\mathcal R}\longrightarrow FI_{\mathcal R}$ preserves colimits. 

\smallskip
(b) For any $\mathscr V\in FI_{\mathcal R}$, we have $\mathscr V=0$ if and only if $H_0(\mathscr V)=0$.

\smallskip
(c) A morphism $\tilde f$ in $FI_{\mathcal R}$ is an epimorphism if and only if $H_0(\tilde f)$
is an epimorphism. 

\end{lem}

\begin{proof}
Part (a) follows from the fact that $H_0$ is defined in \eqref{h0} using cokernels. For (b), suppose we have $\mathscr V\ne 0$ in $FI_{\mathcal R}$ such that $H_0(\mathscr V)=0$. Let
$n$ be the smallest integer $\geq 0$ such that $\mathscr V(n)\ne 0$. Then, for each finite set $T$ with $|T|<n$, we have $\mathscr V(T)=0$ and it follows from 
\eqref{h0} that $H_0(\mathscr V)(n)=\mathscr V(n)$. This yields $\mathscr V(n)=0$, which is a contradiction. This proves (b). Part (c) follows by using (a) and applying the result
of (b) to the cokernel of $\tilde f$. 
\end{proof}

We now consider a functor
\begin{equation}\label{gr}
Gr: FI_{\mathcal R}\longrightarrow Mod-\mathcal R\qquad \mathscr V\mapsto \underset{n\geq 0}{\bigoplus}\mathscr V(n)
\end{equation} It is clear that $Gr(\mathscr V)=0$ if and only if $\mathscr V=0$. We also note that $Gr$ preserves cokernels, kernels and coproducts. 

\begin{thm}\label{P2.6} Let $\mathscr V\in FI_{\mathcal R}$. Then, the following are equivalent:

\smallskip
(a) $\mathscr V$ is finitely generated in $FI_{\mathcal R}$.

\smallskip
(b) $H_0(\mathscr V)$ is finitely generated in $FI_{\mathcal R}$.

\smallskip
(c) $Gr(H_0(\mathscr V))=\underset{n\geq 0}{\bigoplus}H_0(\mathscr V)_n$ is finitely generated in $Mod-\mathcal R$. 

\end{thm}

\begin{proof} (a) $\Rightarrow (b)$ : By definition, $H_0(\mathscr V)$ is a quotient of $\mathscr V$. If $\mathscr V$ is finitely generated, so is its quotient
$H_0(\mathscr V)$.

\smallskip
(b) $\Rightarrow$ (a) : By Proposition \ref{L2.4}(b), we know that the collection $\{_d\mathscr M_r\}_{r\in \mathcal R,d\geq 0}$ is a set of generators for the Grothendieck category $FI_{\mathcal R}$. This gives us an epimorphism 
\begin{equation}
\underset{i\in I}{\bigoplus}\textrm{ } {_{d_i}\mathscr M_{r_i}}\longrightarrow \mathscr V
\end{equation} for some   collection $\{(d_i,r_i)\}_{i\in I}$ with each $d_i\geq 0$ and $r_i\in\mathcal R$. By Lemma \ref{Lem2.5}(c), applying the functor $H_0$ induces an epimorphism 
$\underset{i\in I}{\bigoplus}\textrm{ } H_0({_{d_i}\mathscr M_{r_i}})\longrightarrow H_0(\mathscr V)$. Since $H_0(\mathscr V)$ is finitely generated, it follows that there
is a finite subset $J\subseteq I$ such that $\underset{i\in J}{\bigoplus}\textrm{ } H_0({_{d_i}\mathscr M_{r_i}})\longrightarrow H_0(\mathscr V)$ is an epimorphism. Applying
Lemma \ref{Lem2.5}(c) again, we see that $\underset{i\in J}{\bigoplus}\textrm{ } {_{d_i}\mathscr M_{r_i}}\longrightarrow \mathscr V$ is an epimorphism. Each $_{d_i}\mathscr M_{r_i}$
is finitely generated by Proposition \ref{L2.4}(a) and since $J$ is finite, the result follows. 

\smallskip
(b) $\Rightarrow$ (c) : Since $H_0(\mathscr V)$ is finitely generated and $H_0$ is an idempotent functor, it follows from Proposition \ref{L2.4}(c) that there is an epimorphism 
$\underset{i\in I}{\bigoplus}\textrm{ } H_0({_{d_i}\mathscr M_{r_i}})\longrightarrow H_0(\mathscr V)$ for a finite set $I$.  It suffices therefore to show that $\underset{n\geq 0}{\bigoplus}H_0({_d}\mathcal M_r)_n$ is finitely generated in  $Mod-\mathcal R$  for each $d\geq 0$ and each $r\in \mathcal R$. Since $H_r$ is finitely generated in $Mod-\mathcal R$
for each $r\in \mathcal R$, it is clear from the definitions in \eqref{mdr} and \eqref{h0} that each $H_0({_d}\mathcal M_r)_n$ is finitely generated in  $Mod-\mathcal R$. We also
notice that for $n\geq d+1$, every morphism $[d]\longrightarrow [n]$ in $FI$ factors through a subset of $[n]$ of cardinality $\leq n-1$. The quotient in \eqref{h0} now shows
that $H_0({_d}\mathcal M_r)_n=0$ for $n\geq d+1$. This proves the result. 

\smallskip
(c) $\Rightarrow$ (b) : Since $\{_d\mathscr M_r\}_{r\in \mathcal R,d\geq 0}$ is a set of generators for the category $FI_{\mathcal R}$, we must have an epimorphism 
$
\underset{i\in I}{\bigoplus}\textrm{ } {_{d_i}\mathscr M_{r_i}}\longrightarrow H_0(\mathscr V)
$. Since $Gr(H_0(\mathscr V))$ is finitely generated, we can find some finite subset $J\subseteq I$ such that $Coker\left(\underset{i\in J}{\bigoplus}\textrm{ } Gr({_{d_i}\mathscr M_{r_i}})\longrightarrow Gr(H_0(\mathscr V))\right)=0$. It follows that $\underset{i\in J}{\bigoplus}\textrm{ }  {_{d_i}\mathscr M_{r_i}}\longrightarrow  H_0(\mathscr V)$
is an epimorphism. 
\end{proof}

\begin{thm}\label{P2.7} Let $\mathscr V\in FI_{\mathcal R}$ and fix $d\geq 0$. Then, the following are equivalent:

\smallskip
(a) $\mathscr V$ is generated in degree $\leq d$.

\smallskip
(b) $H_0(\mathscr V)$ is  generated in degree $\leq d$.

\smallskip
(c) $H_0(\mathscr V)_n=0$ for $n>d$. 

\end{thm}

\begin{proof} (a) $\Rightarrow (b)$ : By definition, $H_0(\mathscr V)$ is a quotient of $\mathscr V$. Hence, this is clear from Proposition \ref{L2.4}(e).

\smallskip
(b) $\Rightarrow$ (c) : By reasoning similar to the  proof of (b) $\Rightarrow$ (c) in Proposition \ref{P2.6}, it suffices to show that $H_0({_{d'}}\mathcal M_{r'})_n=0$ 
when $d'\leq d$ and $n>d$. This latter fact has also been established in the proof of Proposition \ref{P2.6}.

\smallskip
(c) $\Rightarrow$ (a) : Since $\{_d\mathscr M_r\}_{r\in \mathcal R,d\geq 0}$ is a set of generators for the category $FI_{\mathcal R}$, we must have an epimorphism
\begin{equation}\label{eqr2.10} 
\tilde e:\underset{i\in I}{\bigoplus}\textrm{ } {_{d_i}\mathscr M_{r_i}}\longrightarrow \mathscr V
\end{equation}  We restrict to all pairs $(d_i,r_i)_{i\in I}$ with $d_i\leq d$ and consider the induced morphism
\begin{equation}
\tilde f: \underset{i\in I, d_i\leq d}{\bigoplus}\textrm{ } {_{d_i}\mathscr M_{r_i}}\longrightarrow \mathscr V
\end{equation} We claim that $\tilde f$ is an epimorphism. By Lemma \ref{Lem2.5}, it suffices to show that $H_0(\tilde f)$ is an epimorphism, i.e., 
each $H_0(\tilde f)_n$ is an epimorphism. For $n>d$ we have
\begin{equation}
H_0(\tilde f)_n: \underset{i\in I,d_i\leq d}{\bigoplus}\textrm{ } H_0\left( {_{d_i}\mathscr M_{r_i}}\right)_n\longrightarrow H_0(\mathscr V)_n=0
\end{equation} which must be an epimorphism. We now consider $n\leq d$ and examine the epimorphism
\begin{equation}\label{eqr2.13} 
H_0(\tilde e)_n:\underset{i\in I}{\bigoplus}\textrm{ }H_0( {_{d_i}\mathscr M_{r_i}})_n\longrightarrow H_0(\mathscr V)_n
\end{equation}
induced by \eqref{eqr2.10}. By definition, $H_0( {_{d_i}\mathscr M_{r_i}})_n$ is a quotient of $H_{r_i}^{([d_i],[n])}$. For any $d_i>d$, we must therefore have $H_0( {_{d_i}\mathscr M_{r_i}})_n=0$ since $n\leq d< d_i$. Hence, $H_0(\tilde f)_n=H_0(\tilde e)_n$ is an epimorphism for $n\leq d$  and the result follows. 

\end{proof}

We now recall some generalities on objects in a Grothendieck abelian category $\mathcal A$ that we will use throughout this paper (see, for instance, \cite{AR}, \cite{Bo} and \cite{Pop})

\smallskip
(a) An object $X$ in $\mathcal A$ is said to be finitely generated if the functor $Hom_{\mathcal A}(X,\_\_):\mathcal A\longrightarrow \mathbf{Ab}$ preserves filtered colimits of monomorphisms.

\smallskip
(b) An object $X$ in $\mathcal A$ is said to be finitely presented if the functor $Hom_{\mathcal A}(X,\_\_):\mathcal A\longrightarrow \mathbf{Ab}$ preserves filtered colimits. 

\smallskip
(c) An object $Y$ in $\mathcal A$ is said to be noetherian if every subobject is finitely generated.

\smallskip
(d) The category $\mathcal A$ is said to be locally noetherian if it has a set of noetherian generators.

\smallskip
In a locally noetherian Grothendieck category $\mathcal A$, the finitely generated objects coincide with the finitely presented objects  (see, for instance, \cite[Chapter 5.8]{Pop}) as well as with the noetherian objects. Further, the full subcategory of finitely generated objects in $\mathcal A$
forms an abelian category, which we denote by $\mathcal A^{fg}$. 

\smallskip
We conclude this section by recalling the following result.

\begin{Thm}\label{Th2.8} (see \cite[Theorem A]{Four1}, \cite{Djament},  \cite{PS}, \cite{SS} and  \cite[Theorem 9.1]{Krause1}) Let $\mathcal A$ be a locally noetherian Grothendieck category. Then, the category $Fun(FI,\mathcal A)$ of functors
from $FI$ to $\mathcal A$ is  locally noetherian. 

\end{Thm} 

In particular, if $\mathcal R$ is a small preadditive category such that $Mod-\mathcal R$ is locally noetherian, it follows from Theorem \ref{Th2.8} that the category $FI_{\mathcal R}$ is
locally noetherian. In that case, if $\mathscr V$ is a finitely generated $FI$-module over $\mathcal R$, any submodule of $\mathscr V$ is finitely generated.

\section{Torsion theories and the positive shift functor}

\medskip
We recall that a torsion theory $\tau=(\mathcal T,\mathcal F)$ on an abelian category $\mathcal A$ consists of a pair of full and replete subcategories $\mathcal T$ and $\mathcal F$
of $\mathcal A$ such that $Hom_{\mathcal A}(T,F)=0$ for any $T\in\mathcal T$, $F\in \mathcal F$ and for any object $X\in \mathcal A$ there exists a short exact sequence
\begin{equation*}
\begin{CD}
0 @>>> T(X)@>>> X@>>> F(X) @>>> 0
\end{CD}
\end{equation*}
with $T(X)\in \mathcal T$, $F(X)\in \mathcal F$ (see, for instance, \cite[$\S$ I.1]{BeRe}). 

\smallskip
Accordingly, we start with a torsion theory $\tau=(\mathcal T,\mathcal F)$ on $Mod-\mathcal R$. Let $\mathcal R$ be a small preadditive category such that
$Mod-\mathcal R$ is locally noetherian. From Theorem \ref{Th2.8}, we know that the category $FI_{\mathcal R}$ is locally noetherian.  We will show how to extend $\tau$ to a torsion theory
$(\overline{\mathcal T},\overline{\mathcal F})$ on the abelian category $FI_{\mathcal R}^{fg}$ of finitely generated $FI$-modules over $\mathcal R$. We let $\overline{\mathcal T}$ be the full subcategory of $FI_{\mathcal R}^{fg}$
defined by setting
\begin{equation}\label{trskc}
Ob(\overline{\mathcal T}):=\{\mbox{$\mathscr V\in Ob(FI_{\mathcal R}^{fg})$ $\vert$ $\mathscr V_n\in \mathcal T$ for $n\gg 0$}\}
\end{equation} We need to show that $\overline{\mathcal T}$ is a torsion class in $FI_{\mathcal R}^{fg}$. If an abelian category is complete and cocomplete, it is well known
(see, for instance, \cite[$\S$ I.1]{BeRe}) that any full subcategory closed under quotients, extensions and arbitrary coproducts must be a torsion class. However, $FI_{\mathcal R}^{fg}$ being the subcategory of finitely generated $FI$-modules, does not contain arbitrary coproducts. As such, in order to identify torsion classes in $FI_{\mathcal R}^{fg}$, we will use the following
simple result from \cite{AB1}. 

\begin{thm}\label{xwP3.1} (see \cite[Proposition 4.8]{AB1}) Let $\mathcal B$ be an abelian category such that every object in $\mathcal B$ is noetherian. Let $\mathcal C\subseteq \mathcal B$ be a
full and replete subcategory that is closed under extensions and quotients. Let $\mathcal C^\perp\subseteq \mathcal B$ be the full subcategory
given by
\begin{equation*}Ob(\mathcal C^\perp) := \{\mbox{$N\in \mathcal B$ $\vert$ $ Hom_{\mathcal B}(C, N) = 0$ for all $C\in \mathcal C$}\}
\end{equation*}
Then $(\mathcal C, \mathcal C^\perp)$
 is a torsion pair on $\mathcal B$.

\end{thm}

\begin{thm}\label{xxP3.2} Let $\mathcal R$ be a small preadditive category such that
$Mod-\mathcal R$ is locally noetherian. Let $\mathcal T$ be a torsion class on $Mod-\mathcal R$. Then, $\overline{\mathcal T}$ is a torsion 
class in the category $FI_{\mathcal R}^{fg}$ of finitely generated $FI$-modules over $\mathcal R$. Additionally, if $\mathcal T$ is a hereditary torsion
class, so is $\overline{\mathcal T}$. 
\end{thm}

\begin{proof} Since $FI_{\mathcal R}$ is locally noetherian, it follows that every object in the category $FI_{\mathcal R}^{fg}$ is noetherian. Applying Proposition 
\ref{xwP3.1}, it suffices to show that $\overline{\mathcal T}$ is closed under extensions and quotients. Accordingly, if 
$0\longrightarrow\mathscr V'\longrightarrow\mathscr V\longrightarrow \mathscr V''\longrightarrow 0$ is a short exact sequence with $\mathscr V'$, $\mathscr V''\in 
\overline{\mathcal T}$, we can choose $N$ large enough so that $\mathscr V'_n$, $\mathscr V''_n\in\mathcal T$ for all $n>N$. We have short exact sequences
\begin{equation}
0\longrightarrow\mathscr V'_n\longrightarrow\mathscr V_n\longrightarrow \mathscr V''_n\longrightarrow 0
\end{equation} Since $\mathcal T$ is closed under extensions, it now follows that $\mathscr V_n\in\mathcal T$ for all $n>N$. Hence, $\mathscr V\in \overline{\mathcal T}$. 

\smallskip
On the other hand, if $\mathscr V'\longrightarrow \mathscr V$ is an epimorphism with $\mathscr V'\in \overline{\mathcal T}$, we know that since $\mathcal T$ is closed under quotients,
we must have $\mathscr V_n\in \mathcal T$ for $n\gg 0$. This gives $\mathscr V\in \overline{\mathcal T}$. By similar reasoning, it is clear that if $\mathcal T$ is a hereditary torsion
class (i.e., closed under subobjects), so is $\overline{\mathcal T}$. 

\end{proof}

\smallskip Given $\mathscr V\in FI_{\mathcal R}^{fg}$, we would like to obtain an explicit description for its torsion subobject in $\overline{\mathcal T}$. For this, we will need
to consider (positive) `shift functors' on the category $FI_{\mathcal R}$ in a manner analogous to \cite[$\S$ 2.1]{Four1}. 
For each $a\geq 0$, we fix a set $[-a]$ of cardinality $a$. Then, the category $FI$ is equipped with a shift functor
\begin{equation}
{\mathbb S}^a: FI\longrightarrow FI \qquad S\mapsto S\sqcup [-a]
\end{equation} formed by taking the disjoint union with $[-a]$. For a morphism $\phi:S\longrightarrow T$ in $FI$, ${\mathbb S}^a(\phi)$ is obtained by extending $\phi$ with
the identity on $[-a]$. Then, ${\mathbb S}^a$ induces a ``positive shift functor'' on $FI_{\mathcal R}$, which we continue to denote by ${\mathbb S}^a$
\begin{equation}
{\mathbb S}^a: FI_{\mathcal R}\longrightarrow FI_{\mathcal R}\qquad \mathscr  V\mapsto \mathscr V\circ {\mathbb S}^a
\end{equation} It is immediate that ${\mathbb S}^a$ preserves all limits and colimits. It is also clear that ${\mathbb S}^a$ does not depend on the choice of the set $[-a]$ of cardinality $a$. Before we proceed further, we will collect some basic properties of the functor ${\mathbb S}^a$. 

\begin{thm}\label{P3.1} Fix $a\geq 0$. If $\mathscr V\in FI_{\mathcal R}$  is generated in degree $\leq d$, then ${\mathbb S}^a(\mathscr V)$ is also generated in degree $\leq d$. 
\end{thm}

\begin{proof}
Since ${\mathbb S}^a$ preserves coproducts and epimorphisms, it follows from the `if and only if' condition
in Proposition \ref{L2.4}(e) that it suffices to prove the result for $\mathscr V={_{d'}\mathscr M_{r}}$ with $d'\leq d$.

\smallskip
Given a finite set $S$, we notice easily that
\begin{equation}
([d'],{\mathbb S}^a(S))= 
\underset{j=0}{\overset{a}{\bigcup}}\textrm{ }([d'-j],S)\times ([j],[-a])
\end{equation} Therefore, we obtain
\begin{equation}\label{eq3.4}
{\mathbb S}^a({_{d'}\mathcal M_r})= 
\underset{j=0}{\overset{a}{\bigoplus}}\textrm{ }{_{d'-j}\mathscr M_r}^{([j],[-a])}
\end{equation} 
\end{proof}

\begin{cor}\label{C3.15} For any $d\geq 0$ and $r\in \mathcal R$, we have ${\mathbb S}^a({_{d}\mathscr M_{r}})={_{d}\mathscr M_{r}}\oplus {_{d}\mathscr N_{r}}$, where $ {_{d}\mathscr N_{r}}$ is finitely generated in degree $\leq d-1$. 
\end{cor}

\begin{proof}
This is clear from Proposition \ref{L2.4} and the expression in \eqref{eq3.4}. 
\end{proof}

\begin{lem}\label{L3.2} Fix $a\geq 0$. Then, for every $\mathscr V\in FI_{\mathcal R}$ and every $n\geq 0$, there is an epimorphism $H_0({\mathbb S}^a(\mathscr V))_n\longrightarrow 
H_0(\mathscr V)_{n+a}$ in $Mod-\mathcal R$.
\end{lem}

\begin{proof}
By the definition in \eqref{h0}, we have
\begin{equation}\label{eq3.5}
\begin{array}{ll}
H_0({\mathbb S}^a(\mathscr V))_n&=Coker\left(\underset{
\mbox{\tiny $\begin{array}{c}\phi:T\hookrightarrow [n] \\  |T|<n \\
\end{array}$}}{\bigoplus}
{\mathbb S}^a(\mathscr V)(T)\longrightarrow{\mathbb S}^a(\mathscr V)_n\right)\\&=Coker\left(\underset{
\mbox{\tiny $\begin{array}{c}\phi\sqcup 1_{[-a]}:(T\sqcup [-a])\hookrightarrow [n+a] \\  |T|<n \\
\end{array}$}}{\bigoplus}
\mathscr V(T\sqcup [-a])\longrightarrow \mathscr V_{n+a}\right)\\
\end{array}
\end{equation} The morphism $\underset{
\mbox{\tiny $\begin{array}{c}\phi\sqcup 1_{[-a]}:(T\sqcup [-a])\hookrightarrow [n+a] \\  |T|<n \\
\end{array}$}}{\bigoplus}
\mathscr V(T\sqcup [-a])\longrightarrow \mathscr V_{n+a}$ appearing in \eqref{eq3.5} factors through the canonical morphism $\underset{
\mbox{\tiny $\begin{array}{c}\phi:T\hookrightarrow [n+a] \\  |T|<n+a \\
\end{array}$}}{\bigoplus}
\mathscr V(T)\longrightarrow\mathscr V_{n+a}$ which gives us a factorization
\begin{equation}\label{eq3.6}
\mathscr V_{n+a}\longrightarrow H_0({\mathbb S}^a(\mathscr V))_n\longrightarrow H_0(\mathscr V)_{n+a}
\end{equation} of the canonical epimorphism $\mathscr V_{n+a}\longrightarrow H_0(\mathscr V)_{n+a}$. It follows that $ H_0({\mathbb S}^a(\mathscr V))_n\longrightarrow H_0(\mathscr V)_{n+a}$ is an epimorphism.
\end{proof}

\begin{thm}\label{P3.3} Fix $a\geq 0$. Suppose that $\mathscr V\in FI_{\mathcal R}$ is such that ${\mathbb S}^a(\mathscr V)$ is generated in degree $\leq d$. Then, $\mathscr V$ is generated
in degree $\leq a+d$. 

\end{thm}

\begin{proof} From Lemma \ref{L3.2}, it is clear that  $H_0({\mathbb S}^a(\mathscr V))_n=0$ $\Rightarrow$ $H_0(\mathscr V)_{n+a}=0$. The result is now a consequence of the equivalent
statements in Proposition \ref{P2.7}. 

\end{proof}

We now return to the torsion theory $\tau=(\mathcal T,\mathcal F)$ on $Mod-\mathcal R$. For any object $ P$ in $Mod-\mathcal R$, we denote its torsion
subobject by $\mathcal T(P)$.  For any $\mathscr V\in FI_{\mathcal R}$ and any finite set $S$, the canonical inclusion $S\hookrightarrow S\sqcup [-a]$ induces a morphism $\mathscr V(S)\longrightarrow \mathscr V(S\sqcup [-a])$  and hence 
a morphism $\psi_a^{\mathscr V}:\mathscr V\longrightarrow {\mathbb S}^a\mathscr V$ in $FI_{\mathcal R}$. 
We now set
\begin{equation}\label{tors3}
\overline{\mathcal T}(\mathscr V)(S):=\underset{a\geq 0}{colim}\textrm{ }lim\left(
\begin{CD}\mathscr V(S)@>\psi_a^{\mathscr V}(S)>>({\mathbb S}^a\mathscr V)(S)@<<<\mathcal T( ({\mathbb S}^a\mathscr V)(S))
\end{CD}\right)
\end{equation} for each finite set $S$. It is clear that $\overline{\mathcal T}(\mathscr V)$ is an $FI$-module and that $\overline{\mathcal T}(\mathscr V)\subseteq 
\mathscr V$. 

\begin{lem}\label{Lem3.7k} Let $\mathcal R$ be  such that
$Mod-\mathcal R$ is locally noetherian. Let $\tau=(\mathcal T,\mathcal F)$   be a hereditary torsion theory on $Mod-\mathcal R$. Then, for any $\mathscr V\in FI_{\mathcal R}^{fg}$, the subobject $\overline{\mathcal T}(\mathscr V)$ belongs to the torsion
class $\overline{\mathcal T}$. 
\end{lem}

\begin{proof}
Since $\mathscr V$ is finitely generated and $\overline{\mathcal T}(\mathscr V)\subseteq 
\mathscr V$, we know that $\overline{\mathcal T}(\mathscr V)$ is finitely generated and hence noetherian. As such, the increasing chain appearing in the definition of
$\overline{\mathcal T}(\mathscr V)$ in \eqref{tors3} must be stationary. In other words, we can find $a>0$ such that
\begin{equation}\label{pa3.10}
\begin{array}{ll}
\overline{\mathcal T}(\mathscr V)(S)&=lim\left(
\begin{CD}\mathscr V(S)@>\psi_a^{\mathscr V}(S)>>\mathscr V(S\sqcup [-a])@<<<\mathcal T( \mathscr V(S\sqcup [-a]))
\end{CD}\right)\\
&=lim\left(
\begin{CD}\mathscr V(S)@>\psi_b^{\mathscr V}(S)>>\mathscr V(S\sqcup [-b])@<<<\mathcal T( \mathscr V(S\sqcup [-b]))
\end{CD}\right)\\
\end{array}
\end{equation} for every $b\geq a$ and every finite set $S$. For the sake of convenience, we put $\mathscr W:=\overline{\mathcal T}(\mathscr V)$. The morphism $
\begin{CD}\mathscr W(S) @>>> \mathscr V(S) @>\psi_b^{\mathscr V}(S)>> \mathscr V(S\sqcup [-b])\end{CD}$ factors through
$\psi_b^{\mathscr W}:\mathscr W(S) \longrightarrow \mathscr W(S\sqcup [-b])$ as well as the subobject $\mathcal T( \mathscr V(S\sqcup [-b]))\subseteq 
\mathscr V(S\sqcup [-b])$. Since $\mathscr W(S\sqcup [-b])\subseteq \mathscr V(S\sqcup [-b])$ and $\tau$ is hereditary, it follows that 
\begin{equation}\label{pa3.11}
Im(\psi_b^{\mathscr W}(S):\mathscr W(S) \longrightarrow \mathscr W(S\sqcup [-b]))\in \mathcal T
\end{equation} We now consider any morphism $\phi:S\longrightarrow S'$ in $FI$ with $|S'|-|S|=b\geq a$. Choosing a bijection between $S\sqcup [-b]$ and $S'$, we obtain  a commutative
diagram 
\begin{equation}\label{pa3.12}
\begin{tikzcd}
{} & \mathscr W(S\sqcup [-b]) \arrow{dr}{\cong} \\
\mathscr W(S) \arrow{ur}{\psi_b^{\mathscr W}(S)} \arrow{rr}{\mathscr W(\phi)} && \mathscr W(S')
\end{tikzcd}
\end{equation} Combining \eqref{pa3.11} and \eqref{pa3.12}, we see that
\begin{equation}\label{pa3.13}
Im(\mathscr W(\phi):\mathscr W(S) \longrightarrow \mathscr W(S'))\in \mathcal T\qquad \forall\textrm{ }\phi:S\longrightarrow S', \textrm{ }|S'|-|S|=b\geq a
\end{equation} Since $\mathscr W$ is finitely generated, we can choose some $d$ such that $\mathscr W$ is finitely generated in degree $d$. By Proposition \ref{P2.7}, we see that
$H_0(\mathscr W)_n=0$ for $n>d$. From the definition of $H_0(\mathscr W)$ in \eqref{h0}, it is clear that
\begin{equation}\label{pa3.14}
\mathscr W_n=\left(\underset{
\mbox{\tiny $\begin{array}{c}\phi:S\hookrightarrow [n] \\  |S|\leq d \\
\end{array}$}}{\sum}Im(\mathscr W(\phi):\mathscr W(S)\longrightarrow \mathscr W_n)\right)\qquad \forall\textrm{ }n>d
\end{equation} Combining \eqref{pa3.13} and \eqref{pa3.14}, we see that for $n>a+d$, we must have $\mathscr W_n\in \mathcal T$. Hence, $\mathscr W\in \overline{\mathcal T}$. 
\end{proof}

\begin{Thm}\label{P3.8cf} Let $\mathcal R$ be  such that
$Mod-\mathcal R$ is locally noetherian.  Let $\tau=(\mathcal T,\mathcal F)$   be a hereditary torsion theory on $Mod-\mathcal R$.  Then, for any $\mathscr V\in FI_{\mathcal R}^{fg}$, the torsion subobject of $\mathscr V$ with respect to the torsion
class $\overline{\mathcal T}$ is given by $\overline{\mathcal T}(\mathscr V)$. 
\end{Thm}

\begin{proof} We set $\mathscr W=\overline{\mathcal T}(\mathscr V)$ and maintain the notation from the proof of Lemma \ref{Lem3.7k}. Then, we have $a>0$ such that
\begin{equation}\label{tx3.15}
\begin{CD}
\mathscr W(S)@>>> {\mathcal T}(\mathscr V(S\sqcup [-b]))\\
@VVV @VVV \\
\mathscr V(S) @>\psi_b^{\mathscr V}(S)>> \mathscr V(S\sqcup [-b])\\
\end{CD}
\end{equation} is a fiber square for each $b\geq a$ and each finite set $S$. Now let $\mathscr W'\subseteq \mathscr V$ be such that $\mathscr W'\in \overline{\mathcal T}$. Then, there exists $N$ such that $\mathscr W'_n\in \mathcal T$ for all $n\geq N$. For $n\geq N+a$, we consider the commutative diagram
\begin{equation}\label{tx3.16}
\begin{CD}
\mathscr W'(S)@>\psi_n^{\mathscr W'}(S)>> \mathscr W'(S\sqcup [-n]))\\
@VVV @VVV \\
\mathscr V(S) @>\psi_n^{\mathscr V}(S)>> \mathscr V(S\sqcup [-n])\\
\end{CD}
\end{equation}
Then $ \mathscr W'(S\sqcup [-n]))\in \mathcal T$ and it follows that the composed morphism $\mathscr W'(S)\longrightarrow \mathscr V(S\sqcup [-n])$ appearing
in \eqref{tx3.16} factors through $ {\mathcal T}(\mathscr V(S\sqcup [-n]))$. From the fiber square \eqref{tx3.15}, it now follows that the inclusion 
$\mathscr W'(S)\hookrightarrow \mathscr V(S)$ factors through a morphism  $\mathscr W'(S)
\longrightarrow \mathscr W(S)$. It follows that $\mathscr  W'\subseteq \mathscr W$. 

\smallskip
We have shown in Proposition \ref{xxP3.2} that $\overline{\mathcal T}$ is a torsion class in $FI_{\mathcal R}^{fg}$. From Lemma \ref{Lem3.7k} we already know that $\mathscr W\in \overline{\mathcal T}$. The reasoning above shows that $\mathscr W=\overline{\mathcal T}(\mathscr V)$ contains all torsion subobjects of $\mathscr V$ and the result follows. 
\end{proof}

We will now apply similar methods to study torsion theories in the subcategory of what we call ``shift finitely generated $FI$-modules.''

\begin{defn}\label{D3.9k} Let $\mathscr V\in FI_{\mathcal R}$. Then, we will say that $\mathscr V$ is  shift finitely generated if there exists $d\geq 0$ such that $\mathbb S^d\mathscr V$ is finitely generated. The full subcategory of shift finitely generated objects will be denoted by $FI^{sfg}_{\mathcal R}$. 
\end{defn} 

\begin{lem}\label{L3.10tp} Let $\mathscr V\in FI_{\mathcal R}$.

\smallskip
(a) If $d\geq 0$ is such that $\mathbb S^d\mathscr V$ is finitely generated, so is $\mathbb S^e\mathscr V$ for any $e\geq d$.

\smallskip
(b) If $\mathscr V$ is shift finitely generated, so is $\mathbb S^a\mathscr V$ for any $a\geq 0$. 

\smallskip
(c) Any $\mathscr V\in FI^{sfg}_{\mathcal R}$ is generated in finite degree. 

\end{lem} 

\begin{proof} Since  $\mathbb S^d\mathscr V$ is finitely generated, we can choose  an epimorphism of the form $\underset{i=1}{\overset{k}{\bigoplus}}\textrm{ } {_{d_i}\mathscr M_{r_i}}\longrightarrow \mathbb S^d\mathscr V$. For $e\geq d$, this induces an epimorphism $\mathbb S^{e-d}\left(\underset{i=1}{\overset{k}{\bigoplus}}\textrm{ } {_{d_i}\mathscr M_{r_i}}\right)\longrightarrow \mathbb S^e\mathscr V$. From Corollary \ref{C3.15}, we know that each $\mathbb S^{e-d} {_{d_i}\mathscr M_{r_i}}$ is finitely generated. This proves
(a). The result of (b) is clear from (a). For (c), we proceed as follows: if $\mathbb S^d\mathscr V$ is finitely generated, it follows from Proposition \ref{P2.6} that $H_0(\mathbb S^d\mathscr V)_n=0$ for $n\gg 0$. Then, the epimorphism $H_0(\mathbb S^d\mathscr V)_n\longrightarrow H_0(\mathscr V)_{n+d}$ in Lemma \ref{L3.2}  shows that $H_0(\mathscr V)_m=0$ for 
$m\gg0$. It now follows from Proposition \ref{P2.7} that $\mathscr V$ is generated in finite degree. 

\end{proof}

\begin{thm}\label{P3.11} Suppose that $Mod-\mathcal R$ is locally noetherian. Then $FI^{sfg}_{\mathcal R}$ is a Serre subcategory of $FI_{\mathcal R}$, i.e., it is closed
under subobjects, quotients and extensions.
\end{thm}

\begin{proof} Let $0\longrightarrow \mathscr V'\longrightarrow\mathscr V\longrightarrow \mathscr V''\longrightarrow 0$ be a short exact sequence in $FI_{\mathcal R}$. Since $\mathbb S$ is exact, this gives a short exact sequence 
\begin{equation}\label{3.17yr} 
0\longrightarrow  \mathbb S^d\mathscr V'\longrightarrow \mathbb S^d\mathscr V\longrightarrow \mathbb S^d \mathscr V''\longrightarrow 0
\end{equation} in $FI_{\mathcal R}$ for each $d\geq 0$. Since $Mod-\mathcal R$ is locally noetherian, it is clear from \eqref{3.17yr} and Theorem \ref{Th2.8} that 
$FI^{sfg}_{\mathcal R}$  is closed under quotients and subobjects. It remains to show that  $FI^{sfg}_{\mathcal R}$  is closed under extensions. We suppose that $\mathscr V'$, 
$\mathscr V''\in FI_{\mathcal R}^{sfg}$  and choose $d$ large enough so that $\mathbb S^d\mathscr V'$ and $\mathbb S^d\mathscr V''$ are finitely generated. 

\smallskip
Consequently, we can choose epimorphisms $\mathscr P\longrightarrow \mathbb S^d\mathscr V'$ and $\mathscr Q\longrightarrow \mathbb S^d\mathscr V''$, where $\mathscr P$ and $\mathscr Q$ are finite direct sums
of the family $\{_d\mathscr M_r\}_{r\in \mathcal R,d\geq 0}$ of generators of $FI_{\mathcal R}$. Using Lemma \ref{L2.3}, each object in $\{_d\mathscr M_r\}_{r\in \mathcal R,d\geq 0}$ is projective and hence the epimorphism $\mathscr Q\longrightarrow \mathbb S^d\mathscr V''$ lifts to a morphism $\mathscr Q\longrightarrow \mathbb S^d\mathscr V$. It may be easily verified that the induced morphism $\mathscr P\oplus \mathscr Q\longrightarrow \mathbb S^d\mathscr V$ is an epimorphism and the result follows. 

\end{proof}

Our next objective is to define a torsion theory on $FI^{sfg}_{\mathcal R}$ starting from a torsion theory $\tau=(\mathcal T,\mathcal F)$ on $Mod-\mathcal R$. For this, we need to identify the torsion objects  in $FI^{sfg}_{\mathcal R}$. For finitely generated objects in $FI_{\mathcal R}$, we already have that 
\begin{equation}\label{trskc1}
Ob(\overline{\mathcal T}):=\{\mbox{$\mathscr V\in Ob(FI_{\mathcal R}^{fg})$ $\vert$ $\mathscr V_n\in \mathcal T$ for $n\gg 0$}\}
\end{equation}
as defined in \eqref{trskc}.  In the case of $FI_{\mathcal R}^{sfg}$, we cannot proceed directly as in the proof of Proposition \ref{xxP3.2}, because every object in $FI^{sfg}_{\mathcal R}$ is not necessarily noetherian. For a hereditary torsion theory $\tau=(\mathcal T,\mathcal F)$, we now set
\begin{equation}\label{trskc2}
Ob(\overline{\mathcal T}^{sfg}):=\{\mbox{$\mathscr V\in Ob(FI_{\mathcal R}^{sfg})$ $\vert$ Every finitely generated $\mathscr W\subseteq \mathscr V$ lies in $\overline{\mathcal T}$}\}
\end{equation} Accordingly, we set 
\begin{equation}
Ob(\overline{\mathcal F}^{sfg}):=\{\mbox{$\mathscr V\in Ob(FI_{\mathcal R}^{sfg})$ $\vert$ $Hom(\mathscr W,\mathscr V)=0$ for every $\mathscr W\in \overline{\mathcal T}^{sfg}$ }\}
\end{equation} We will now show that $(\overline{\mathcal T}^{sfg},\overline{\mathcal F}^{sfg})$ defines a torsion theory on $FI^{sfg}_{\mathcal R}$. 

\begin{lem}\label{L3.12}  Suppose that $Mod-\mathcal R$ is locally noetherian and let $\tau=(\mathcal T,\mathcal F)$ be a hereditary torsion theory on $Mod-\mathcal R$. For $\mathscr V\in FI^{sfg}_{\mathcal R}$, let $\mathscr T$ be the sum of all finitely generated subobjects of $\mathscr V$ which lie in $\overline{\mathcal T}$. Then, $\mathscr T\in 
\overline{\mathcal T}^{sfg}$. 
\end{lem} 

\begin{proof} Let $\mathscr T=\sum_{i\in I}\mathscr T_i$, where  $\{\mathscr T_i\}_{i\in I}$ is the collection  of all finitely generated subobjects of $\mathscr V$ which lie in $\overline{\mathcal T}$. Then, we can express $\mathscr T$ as the filtered colimit $\mathscr T=\underset{J\in Fin(I)}{\varinjlim}\textrm{ }\underset{j\in J}{\sum}\mathscr T_j$, where $Fin(I)$ is the collection
of finite subsets of $I$. We now consider some finitely generated $\mathscr T'\subseteq \mathscr T$. Then, there exists some finite $J\subseteq I$ such that $\mathscr T'\subseteq
\underset{j\in J}{\sum}\mathscr T_j$, i.e., we have a monomorphism $\mathscr T'\hookrightarrow Im\left(\underset{j\in J}{\bigoplus}\mathscr T_j\longrightarrow \mathscr V\right)$. Since $\mathcal T$ is a hereditary torsion class, it follows from Proposition \ref{xxP3.2} that $\overline{\mathcal T}$ is closed under extensions, quotients and subobjects. Since each $\mathscr T_j\in \overline{\mathcal T}$, it is now clear
that $\mathscr T'\in \overline{\mathcal T}$. Hence, $\mathscr T\in \overline{\mathcal T}^{sfg}$. 

\end{proof}

\begin{thm}\label{P3.13}  Suppose that $Mod-\mathcal R$ is locally noetherian and let $\tau=(\mathcal T,\mathcal F)$ be a hereditary torsion theory on $Mod-\mathcal R$. Then
$\overline{\mathcal T}^{sfg}$ is a hereditary torsion class in $FI^{sfg}_{\mathcal R}$. 
\end{thm}

\begin{proof} From the definition in \eqref{trskc2}, it is clear that $\overline{\mathcal T}^{sfg}$  is closed under subobjects. We take $\mathscr V\in FI^{sfg}_{\mathcal R}$
and let $\mathscr T\subseteq \mathscr V$ be as in the proof of Lemma \ref{L3.12}. From Lemma \ref{L3.12}, we know that $\mathscr T\in \overline{\mathcal T}^{sfg}$. It suffices
therefore to show that $\mathscr V/\mathscr T\in \overline{\mathcal F}^{sfg}$

\smallskip
We consider therefore a morphism $f:\mathscr X\longrightarrow \mathscr V/\mathscr T$ with $\mathscr X\in \overline{\mathcal T}^{sfg}$. Since $FI_{\mathcal R}$ is locally finitely
generated, we can show that $f=0$ by verifying that $f'=f|_{\mathscr X'}: \mathscr X'\longrightarrow \mathscr V/\mathscr T$ is zero for every finitely generated $\mathscr X'
\subseteq \mathscr X$. First, we note that we can write $Im(f': \mathscr X'\longrightarrow \mathscr V/\mathscr T)=\mathscr Y/\mathscr T$ where $\mathscr T\subseteq 
\mathscr Y\subseteq \mathscr V$.  Since $\mathscr X\in \overline{\mathcal T}^{sfg}$, it follows that $\mathscr X'\in \overline{\mathcal T}$. Since $\overline{\mathcal T}$ is closed
under quotients, it follows that $\mathscr Y/\mathscr T\in \overline{\mathcal T}$. 

\smallskip
We now consider a finitely generated subobject $\mathscr Z\subseteq \mathscr Y$. Since $\overline{\mathcal T}$ is closed under subobjects, we get $\mathscr Y/\mathscr T
\supseteq (\mathscr Z+\mathscr T)/\mathscr T=\mathscr Z/(\mathscr Z\cap \mathscr T)\in \overline{\mathcal T}$. On the other hand, since $FI_{\mathcal R}$ is locally noetherian, 
we know that $\mathscr Z\cap\mathscr T\subseteq \mathscr Z$ must be finitely generated. Now since $\mathscr Z\cap\mathscr T\subseteq \mathscr T$ and $\mathscr T
\in \overline{\mathcal T}^{sfg}$, it follows that $\mathscr Z\cap\mathscr T\in \overline{\mathcal T}$. Since $\overline{\mathcal T}$ is closed under extensions, the short exact sequence
\begin{equation}
0\longrightarrow \mathscr Z\cap\mathscr T\longrightarrow \mathscr Z\longrightarrow \mathscr Z/\mathscr Z\cap\mathscr T\longrightarrow 0
\end{equation}
gives $\mathscr Z\in \overline{\mathcal T}$. From the definition of $\mathscr T$, it now follows that $\mathscr Z\subseteq \mathscr T$. Again since $FI_{\mathcal R}$ is locally finitely
generated, this gives $\mathscr Y\subseteq \mathscr T$. Hence, $f'=0$. This proves the result. 

\end{proof} 

We will now compute an expression for the torsion submodule $\overline{\mathcal T}^{sfg}(\mathscr V)$ of  $\mathscr V\in FI^{sfg}_{\mathcal R}$. This will be done in several steps. We denote by  $fg(\mathscr V)$
the collection of finitely generated subobjects of $\mathscr V$.

\begin{thm}\label{P3.14}  Suppose that $Mod-\mathcal R$ is locally noetherian and let $\tau=(\mathcal T,\mathcal F)$ be a hereditary torsion theory on $Mod-\mathcal R$. 

\smallskip

 (a) Suppose $\mathscr V\in FI^{fg}_{\mathcal R}$. Then, $\overline{\mathcal T}(\mathscr V)=\overline{\mathcal T}^{sfg}(\mathscr V)$.

\smallskip
(b) For $\mathscr V\in FI^{sfg}_{\mathcal R}$, we have $\overline{\mathcal T}^{sfg}(\mathscr V)=\underset{\mathscr V'\in fg(\mathscr V)}{\varinjlim}\textrm{ }\overline{\mathcal T}^{sfg}(\mathscr V')=\underset{\mathscr V'\in fg(\mathscr V)}{\varinjlim}\textrm{ }\overline{\mathcal T}(\mathscr V')$.

\end{thm}

\begin{proof} (a)  It is immediate that $\overline{\mathcal T}(\mathscr V)\subseteq \overline{\mathcal T}^{sfg}(\mathscr V)$.  Since $\mathscr V$ is finitely generated, we know that $\overline{\mathcal T}^{sfg}(\mathscr V)$ is finitely generated. Since $ \overline{\mathcal T}^{sfg}(\mathscr V)\in \overline{\mathcal T}^{sfg}$, it follows that $\overline{\mathcal T}^{sfg}(\mathscr V)\in \overline{\mathcal T}$. Hence,
 $\overline{\mathcal T}^{sfg}(\mathscr V)\subseteq \overline{\mathcal T}(\mathscr V) $ and the result follows.
 
 \smallskip
 (b) It is clear that $\underset{\mathscr V'\in fg(\mathscr V)}{\varinjlim}\textrm{ }\overline{\mathcal T}^{sfg}(\mathscr V')\subseteq
 \overline{\mathcal T}^{sfg}(\mathscr V)$.  Conversely, we consider any finitely generated subobject $\mathscr W\subseteq \overline{\mathcal T}^{sfg}(\mathscr V)$. Then, $\mathscr W\in 
 \overline{\mathcal T}$ and hence $\mathscr W=\overline{\mathcal T}(\mathscr W)=\overline{\mathcal T}^{sfg}(\mathscr W)$. It follows that $\mathscr W\subseteq \underset{\mathscr V'\in fg(\mathscr V)}{\varinjlim}\textrm{ }\overline{\mathcal T}^{sfg}(\mathscr V')$. Since $FI_{\mathcal R}$ is locally finitely generated, we get $ \overline{\mathcal T}^{sfg}(\mathscr V)
 \subseteq \underset{\mathscr V'\in fg(\mathscr V)}{\varinjlim}\textrm{ }\overline{\mathcal T}^{sfg}(\mathscr V')$ and the result follows.

\end{proof}

\begin{lem}\label{L3.15} Let $\mathcal B$ be a locally noetherian Grothendieck category and suppose that $t=(\mathbf T,\mathbf F)$ is a hereditary torsion theory on $\mathcal B$. Let $M\in \mathcal B$
and let $i:N\hookrightarrow M$ be a subobject satisfying the following properties:

\smallskip
(a) $N\in \mathbf T$. 

\smallskip
(b) If $T'\in \mathbf T$ is a finitely generated object, any morphism $f':T'\longrightarrow M$ factors through $N$. 

\smallskip
Then, $N$ is the torsion subobject of $M$, i.e., $N=\mathbf T(M)$. 
\end{lem} 

\begin{proof}
We consider $T\in \mathbf T$ and a morphism $f:T\longrightarrow M$.  Let  $fg(T)$ be the collection of finitely generated subobjects of $T$. For any $T'\in fg(T)$, the induced map 
$f'=f|_{T'}:T'\longrightarrow M$ factors through some $g':T'\longrightarrow N$ as $f'=i\circ g'$. Since $i$ is a monomorphism, this $g'$ is necessarily unique.

\smallskip
If $j:T'\hookrightarrow T''$ is an inclusion with $T', T''\in fg(T)$, we notice that $i\circ g'=f'=f''\circ j=i\circ g''\circ j$. Since $i$ is a monomorphism, this gives $g'=g''\circ j$. These maps
$\{g':T'\longrightarrow N\}_{T'\in fg(T)}$ together induce a morphism from the colimit $g:T=\underset{T'\in fg(T)}{\varinjlim}\textrm{ }T'\longrightarrow N$. We notice that
$(i\circ g)|_{T'}=i\circ g'=f'=f|_{T'}$ for each $T'\in fg(T)$. Since $T=\underset{T'\in fg(T)}{\varinjlim}\textrm{ }T'$, it now follows that $i\circ g=f$. This proves the result. 
\end{proof}

\begin{thm}\label{P3.16}  Let $\mathcal B$ be a locally noetherian Grothendieck category and suppose that $t=(\mathbf T,\mathbf F)$ is a hereditary torsion theory on $\mathcal B$. 
Let $\{M_i\}_{i\in I}$ be a filtered system of objects of $\mathcal B$. Then, we have
\begin{equation}
\underset{i\in I}{\varinjlim}\textrm{ }\mathbf T(M_i)=\mathbf T \left(\underset{i\in I}{\varinjlim}\textrm{ }M_i\right)
\end{equation}

\end{thm} 

\begin{proof} Since torsion classes are always closed under colimits, we know that $\underset{i\in I}{\varinjlim}\textrm{ }\mathbf T(M_i)\in \mathbf T$. Considering the monomorphisms 
$\mathbf T(M_i)\hookrightarrow M_i$, the filtered colimit induces an inclusion $\underset{i\in I}{\varinjlim}\textrm{ }\mathbf T(M_i)\hookrightarrow \underset{i\in I}{\varinjlim}\textrm{ }M_i$. We now consider a finitely generated object $T'\in \mathbf T$ along with a morphism $f':T'\longrightarrow \underset{i\in I}{\varinjlim}\textrm{ }M_i$. Since $\mathcal B$ is locally noetherian, $T'$ is also finitely presented. It follows that $f'$ factors through some $g':T'\longrightarrow M_{i_0}$. Since $T'\in \mathbf T$, $g'$ factors uniquely through 
the torsion subobject $\mathbf T(M_{i_0})$. Hence, $f':T'\longrightarrow \underset{i\in I}{\varinjlim}\textrm{ }M_i$ factors through $\underset{i\in I}{\varinjlim}\textrm{ }\mathbf T(M_i)$. The result now follows from Lemma \ref{L3.15}. 

\end{proof}

\begin{Thm}\label{P3.17cf} Let $\mathcal R$ be  such that
$Mod-\mathcal R$ is locally noetherian. Let $\tau=(\mathcal T,\mathcal F)$ be a hereditary torsion theory on $Mod-\mathcal R$. For any $\mathscr V\in FI_{\mathcal R}^{sfg}$, the torsion subobject of $\mathscr V$ with respect to the torsion
class $\overline{\mathcal T}^{sfg}$ is given by 
\begin{equation}\label{tors317}
\overline{\mathcal T}^{sfg}(\mathscr V)(S):=\underset{a\geq 0}{colim}\textrm{ }lim\left(
\begin{CD}\mathscr V(S)@>\psi_a^{\mathscr V}(S)>>({\mathbb S}^a\mathscr V)(S)@<<<\mathcal T( ({\mathbb S}^a\mathscr V)(S))
\end{CD}\right)
\end{equation} for each finite set $S$.
\end{Thm}

\begin{proof}
From Proposition \ref{P3.14}, we know that  
\begin{equation}\label{eq3.24}\overline{\mathcal T}^{sfg}(\mathscr V)=\underset{\mathscr V'\in fg(\mathscr V)}{\varinjlim}\textrm{ }\overline{\mathcal T}(\mathscr V')
\end{equation} Since each $\mathscr V'\in fg(\mathscr V)$ lies in $FI^{fg}_{\mathcal R}$, it follows from Theorem \ref{P3.8cf} that
\begin{equation}\label{eq3/25}
\overline{\mathcal T}(\mathscr V')(S):=\underset{a\geq 0}{colim}\textrm{ }lim\left(
\begin{CD}\mathscr V'(S)@>\psi_a^{\mathscr V'}(S)>>({\mathbb S}^a\mathscr V')(S)@<<<\mathcal T( ({\mathbb S}^a\mathscr V')(S))
\end{CD}\right)
\end{equation} for each finite set $S$. Since $Mod-{\mathcal R}$ is locally noetherian and  $\mathscr V=\underset{\mathscr V'\in fg(\mathscr V)}{\varinjlim}\textrm{ }\mathscr V'$, it now follows from Proposition \ref{P3.16} that
\begin{equation}\label{eq3.26}
\mathcal T( ({\mathbb S}^a\mathscr V)(S))=\underset{\mathscr V'\in fg(\mathscr V)}{\varinjlim}\textrm{ }\mathcal T( ({\mathbb S}^a\mathscr V')(S))
\end{equation} for each $a\geq 0$. The result of \eqref{tors317} is now clear from \eqref{eq3.24}, \eqref{eq3/25}, \eqref{eq3.26} and the fact that filtered colimits commute with finite limits.
\end{proof}

\section{Torsion closed $\mathbf{FI}$-modules}

\medskip
We continue with $\mathcal R$ being a small preadditive category such that $Mod-\mathcal R$ is locally noetherian. Given a torsion theory $\tau=(\mathcal T,\mathcal F)$
on $Mod-\mathcal R$, we have described the induced torsion class $\overline{\mathcal T}$ on finitely generated $FI$-modules. In this section, we will always suppose that $\mathcal T$
is a hereditary torsion class. Then, from Proposition \ref{xxP3.2}, we know that $\overline{\mathcal T}$ is a hereditary torsion class on $FI^{fg}_{\mathcal R}$. 

\smallskip
 We now consider the full subcategories ${\hat{\mathcal T}}$ and $\{{\hat{\mathcal T^a}}\}_{a\geq 0}$ of $FI_{\mathcal R}$
determined by setting
\begin{equation}\label{4.3ne}
\begin{array}{c}
Ob({\hat{\mathcal T}}):=\{\mbox{$\mathscr V\in Ob(FI_{\mathcal R})$ $\vert$ $\mathscr V_n\in \mathcal T$ for $n\gg 0$}\}\\
Ob({\hat{\mathcal T^a}}):=\{\mbox{$\mathscr V\in Ob(FI_{\mathcal R})$ $\vert$ $\mathscr V_n\in \mathcal T$ for all $n\geq a$}\} \qquad \forall\textrm{ }a\geq 0\\
\end{array}
\end{equation} It is also clear that we have a filtration
\begin{equation}\label{filtne}
{\hat{\mathcal T^0}}\subseteq {\hat{\mathcal T^1}}\subseteq  {\hat{\mathcal T^2}}\subseteq \dots \qquad {\hat{\mathcal T}}=\underset{a\geq 0}{\bigcup}
{\hat{\mathcal T^a}}
\end{equation}
We observe that each ${\hat{\mathcal T^a}}$ is closed under extensions, quotients, subobjects and coproducts, making it a hereditary torsion class
in the category $FI_{\mathcal R}$ (see, for instance, \cite[$\S$ I.1]{BeRe}). However, we notice that ${\hat{\mathcal T}}$ need not be a torsion class in $FI_{\mathcal R}$, because it may not contain arbitrary coproducts. In fact, ${\hat{\mathcal T}}$ is only a Serre subcategory, i.e., it is closed under extensions, subobjects and quotients. The purpose of this section is to construct functors from $FI_{\mathcal R}$ to  ${\hat{\mathcal T^a}}$-closed objects and to ${\hat{\mathcal T}}$-closed objects of  $FI_{\mathcal R}$.  For this, we will first develop some general results on locally noetherian Grothendieck categories.

\begin{defn}\label{D4.1} (see, for instance, \cite[Definition III.2.2]{Hvedri}) Let $\mathcal A$ be a Grothendieck category and let $\mathcal C$ be a Serre subcategory. Then:

\smallskip
(1) A morphism $u:A\longrightarrow B$ in $\mathcal A$ is said to be a $\mathcal C$-isomorphism if both $Ker(u)$ and $Coker(u)$ lie in $\mathcal C$.

\smallskip
(2) An object $L$ in $\mathcal A$ is said to be $\mathcal C$-closed if for every $\mathcal C$-isomorphism $u:A\longrightarrow B$ in $\mathcal A$, the induced morphism 
 $Hom(u,L):Hom(B,L)\longrightarrow Hom(A,L)$ is an isomorphism.
 
 \smallskip
 (3) A morphism $f:A\longrightarrow A_{\mathcal C}$ in $\mathcal A$ is said to be a $\mathcal C$-envelope if $f$ is a $\mathcal C$-isomorphism and $A_{\mathcal C}$ is $\mathcal C$-closed. 

\end{defn}

Let $\mathcal A$ be a Grothendieck category. From now onwards, for $X\in \mathcal A$, we will denote by $fg(X)$ the set of its finitely generated subobjects.

\begin{lem}\label{4Lem1} Let  $\mathcal A$ be a  locally noetherian Grothendieck category and $\mathcal C$ be a Serre subcategory. Suppose that an object $L\in \mathcal A$ has the following property : for any $\mathcal C$-isomorphism $u':A'\longrightarrow B'$ with $A'$, $B'$ finitely generated, the induced morphism  $Hom(u',L):Hom(B',L)\longrightarrow Hom(A',L)$ is an isomorphism.

\smallskip
Then, for any $\mathcal C$-isomorphism $u:A\longrightarrow B$ that is an epimorphism in $\mathcal A$, the induced morphism  $Hom(u,L):Hom(B,L)\longrightarrow Hom(A,L)$ is an isomorphism.
\end{lem}

\begin{proof} We consider a  $\mathcal C$-isomorphism $u:A\longrightarrow B$ that is an epimorphism in $\mathcal A$. By definition, $Coker(u)=0$ and $Ker(u)\in \mathcal C$. Let $A' \subseteq A$ be a finitely generated subobject and let $u':A'\longrightarrow B$ denote the restriction of $u$ to $A'$. Set $B':=Im(u')\subseteq B$. Since $A'$ is finitely generated,
so is its quotient $B'$. Then, $u':A'\longrightarrow B'$ satisfies
$Coker(u')=0$ and $Ker(u')\subseteq Ker(u)\in \mathcal C$. Since $\mathcal C$ is a Serre subcategory, we get $Ker(u')\in \mathcal C$. It follows that $u'$ is a $\mathcal C$-isomorphism.

\smallskip
Using the given property of $L$, we now obtain that the induced morphism  $Hom(u',L):Hom(B',L)\longrightarrow Hom(A',L)$ is an isomorphism. Since $\mathcal A$ is locally finitely generated, we know that $A$ is the filtered colimit over all $A'\in fg(A)$. Since $u$ is an epimorphism, we know that $B$ is the filtered colimit over the corresponding objects $\{B'=Im(u|_{A'}:A'\longrightarrow B)\}_{A'\in fg(A)}$.  It follows that
\begin{equation}
Hom(u,L):Hom(B,L)=\underset{A'\in fg(A)}{\varprojlim}\textrm{ }Hom(B',L)\overset{\cong}{\longrightarrow} \underset{A'\in fg(A)}{\varprojlim}\textrm{ }Hom(A',L)=Hom(A,L)
\end{equation} is an isomorphism. This proves the result. 

\end{proof}

\begin{lem}\label{4Lem2} Let  $\mathcal A$ be a locally noetherian Grothendieck category and $\mathcal C$ be a Serre subcategory. Suppose that an object $L\in \mathcal A$ has the following property : for any $\mathcal C$-isomorphism $u':A'\longrightarrow B'$ with $A'$, $B'$ finitely generated, the induced morphism  $Hom(u',L):Hom(B',L)\longrightarrow Hom(A',L)$ is an isomorphism.

\smallskip
Then, for any $\mathcal C$-isomorphism $u:A\longrightarrow B$ that is a monomorphism in $\mathcal A$, the induced morphism  $Hom(u,L):Hom(B,L)\longrightarrow Hom(A,L)$ is an isomorphism.
\end{lem}

\begin{proof} We consider a  $\mathcal C$-isomorphism $u:A\longrightarrow B$ that is a monomorphism in $\mathcal A$. Then, by definition, $Ker(u)=0$ and $Coker(u)\in \mathcal C$. 
Let $B'\subseteq B$ be a finitely generated subobject and let $u':A':=A\times_BB'\longrightarrow B'$ be the pullback of $u$ along $B'\hookrightarrow B$. Clearly, $Ker(u')=0$. Since $\mathcal A$ is locally noetherian and $A'\subseteq B'$, it follows that $A'$ is finitely generated.

\smallskip We now note that all the squares in the following diagram are pullback squares.
\begin{equation}
\begin{CD}
A' @>u'>> B' @>>> 0 \\
@VVV @VVV @VVV \\
A @>u>> B @>>> B/B'\\
\end{CD}
\end{equation}
It follows that $A'=Ker(A\overset{u}{\longrightarrow}B\longrightarrow B/B')$ and hence we have a monomorphism $A/A'\longrightarrow B/B'$. A simple application of Snake Lemma to the following diagram 
\begin{equation}
\begin{CD}
0@>>> A' @>>> A @>>> A/A'@>>>  0\\
@. @Vu'VV @VuVV @VVV @. \\
0 @>>> B' @>>> B @>>> B/B' @>>> 0\\
\end{CD}
\end{equation}
now gives us the exact sequence $0\longrightarrow Coker(u')\longrightarrow Coker(u)$. Since $\mathcal C$ is a Serre subcategory, we now get $Coker(u')\in \mathcal C$. 
Using the given property of $L$, we now obtain that the induced morphism  $Hom(u',L):Hom(B',L)\longrightarrow Hom(A',L)$ is an isomorphism. Since $\mathcal A$ is locally finitely generated, we know that $B$ is the filtered colimit over all $B'\in fg(B)$. We notice that $A$ is the filtered colimit over the corresponding objects $\{A'=A\times_BB'\}_{B'\in fg(B)}$.  It follows that
\begin{equation}
Hom(u,L):Hom(B,L)=\underset{B'\in fg(B)}{\varprojlim}\textrm{ }Hom(B',L)\overset{\cong}{\longrightarrow} \underset{B'\in fg(B)}{\varprojlim}\textrm{ }Hom(A',L)=Hom(A,L)
\end{equation} is an isomorphism. This proves the result.

\end{proof}

\begin{thm}\label{P4.4}Let  $\mathcal A$ be a locally noetherian Grothendieck category and $\mathcal C$ be a Serre subcategory.  For an object $L\in \mathcal A$, the following are equivalent.

\smallskip
(1) The object $L$ is $\mathcal C$-closed, i.e., for any $\mathcal C$-isomorphism $u:A\longrightarrow B$, the induced morphism  $Hom(u,L):Hom(B,L)\longrightarrow Hom(A,L)$ is an isomorphism.

\smallskip
(2) For any $\mathcal C$-isomorphism $u':A'\longrightarrow B'$ with $A'$, $B'$ finitely generated, the induced morphism  $Hom(u',L):Hom(B',L)\longrightarrow Hom(A',L)$ is an isomorphism.

\end{thm}

\begin{proof}
We only need to show that (2) $\Rightarrow$ (1). Let $u:A\longrightarrow B$ be a $\mathcal C$-isomorphism. Then, we can factor $u$ uniquely as 
$A\overset{f}{\longrightarrow}C\overset{g}{\longrightarrow}B$ where $f$ is an epimorphism and $g$ is a monomorphism. We notice that
\begin{equation}
Ker(f)=Ker(u)\in \mathcal C\qquad Coker(f)=0\qquad Ker(g)=0\qquad Coker(g)=Coker(u)\in\mathcal C
\end{equation} and hence both $f$ and $g$ are $\mathcal C$-isomorphisms. The result is now clear from Lemma \ref{4Lem1} and Lemma \ref{4Lem2}. 
\end{proof}

We now return to $FI$ modules over $\mathcal R$ along with a hereditary torsion theory $\tau=(\mathcal T,\mathcal F)$ on $Mod-\mathcal R$. Accordingly, there is a functor 
$\mathbb E_\tau: Mod-\mathcal R\longrightarrow Mod-\mathcal R$ that takes any $V\in Mod-\mathcal R$ to its torsion envelope $\mathbb E_\tau( V)$. We refer the reader
to \cite[Theorem 2.5]{Gar} for the explicit construction of this functor. Since  $Mod-\mathcal R$  is a locally finitely presented Grothendieck category, we note that hereditary torsion
classes in $Mod-\mathcal R$ are the same as localizing subcategories of $Mod-\mathcal R$ (see, for instance, \cite[Theorem 1.13.5]{Bor}). 

\smallskip By abuse of notation,  we will also denote by
$\mathbb E_\tau$ the functor given by
\begin{equation}\label{envel4}
\mathbb E_\tau: FI_{\mathcal R}\longrightarrow FI_{\mathcal R}\qquad \mathbb E_\tau(\mathscr V)(S):=\mathbb E_\tau(\mathscr V(S))
\end{equation} for any $\mathscr V\in FI_{\mathcal R}$ and any finite set $S$. The canonical morphisms $\mathscr V(S)\longrightarrow \mathbb E_\tau(\mathscr V(S))$ together
induce a morphism $i_{\tau}(\mathscr V):\mathscr V\longrightarrow \mathbb E_\tau(\mathscr V)$ in $FI_{\mathcal R}$. We also observe that from \eqref{envel4} it is clear that
\begin{equation}\label{envel5}
\mathbb E_\tau\mathbb S^a(\mathscr V)=\mathbb S^a\mathbb E_\tau(\mathscr V)\qquad\forall \textrm{ }a\geq 0
\end{equation} We will denote by $Cl(\mathcal T)$ (resp. $Cl(\hat{\mathcal T^a})$, $Cl(\hat{\mathcal T})$) the full subcategory of $Mod-\mathcal R$ (resp. 
$FI_{\mathcal R}$) consisting of closed objects with respect to the Serre subcategory $\mathcal T\subseteq Mod-\mathcal R$ (resp. $\hat{\mathcal T^a}$, $\hat{\mathcal T}
\subseteq FI_{\mathcal R}$). 

\begin{lem}\label{Lem4.5} Let $\mathscr L\in FI_{\mathcal R}$ be such that $\mathscr L(S)$ is $\mathcal T$-closed for each finite set $S$. Then, $\mathscr L$
is $\hat{\mathcal T^0}$-closed.

\end{lem}

\begin{proof}
Let $u:\mathscr A\longrightarrow \mathscr B$ be a $\hat{\mathcal T^0}$-isomorphism in $FI_{\mathcal R}$. Then, by definition, we have $Ker(u)$, $Coker(u)\in \hat{\mathcal T^0}$, i.e., for each finite set $S$,
we must have $Ker(u(S))$, $Coker(u(S))\in \mathcal T$. We consider a morphism $f:\mathscr B\longrightarrow \mathscr L$ in $FI_{\mathcal R}$. If $f\circ u=0$, it follows that
$f(S)\circ u(S)=0$ for each $S\in FI$. Since each $\mathscr L(S)$ is $\mathcal T$-closed, we know that $Hom(u(S),\mathscr L(S)):Hom(\mathscr B(S),\mathscr L(S))
\longrightarrow Hom(\mathscr A(S),\mathscr L(S))$ is an isomorphism. This gives $f(S)=0$ for each $S\in FI$, i.e., $f=0$. 

\smallskip
On the other hand, consider a morphism $g:\mathscr A\longrightarrow \mathscr L$ in $FI_{\mathcal R}$. Each $\mathscr L(S)$ is $\mathcal T$-closed, which gives us a unique morphism 
morphism $f(S):\mathscr B(S)\longrightarrow\mathscr L(S)$ such that $g(S)=f(S)\circ u(S)$. We claim that $\{f(S)\}_{S\in FI}$ gives a morphism $f:\mathscr B
\longrightarrow \mathscr L$, i.e., for any $\phi: S\longrightarrow T$ in $FI$, we have $\mathscr L(\phi)\circ f(S)=f(T)\circ \mathscr B(\phi):\mathscr B(S)
\longrightarrow \mathscr L(T)$. For this, we notice that 
\begin{equation}\label{eq4.10}
f(T)\circ \mathscr B(\phi)\circ u(S) = f(T)\circ u(T)\circ \mathscr A(\phi)=\mathscr L(\phi)\circ f(S)\circ u(S):\mathscr A(S)\longrightarrow \mathscr L(T)
\end{equation} Since $u(S)$ is a $\mathcal T$-isomorphism and $\mathscr L(T)$ is $\mathcal T$-closed, we must have an isomorphism
$Hom(\mathscr B(S),\mathscr L(T))\longrightarrow Hom(\mathscr A(S),\mathscr L(T))$. From \eqref{eq4.10}, it is now clear that $\mathscr L(\phi)\circ f(S)=f(T)\circ \mathscr B(\phi)$.

\smallskip
We have now shown that the induced morphism $FI_{\mathcal R}(u,\mathscr L):FI_{\mathcal R}(\mathscr B,\mathscr L)\longrightarrow FI_{\mathcal R}(\mathscr A,\mathscr L)$ is both a monomorphism and an epimorphism, i.e., an isomorphism. This proves the result.
\end{proof}

\begin{thm}\label{P4.6} Let $\mathscr L\in FI_{\mathcal R}$. Then, $\mathscr L$ is $\hat{\mathcal T^0}$-closed if and only if $\mathscr L(S)$ is $\mathcal T$-closed
for each finite set $S$. 
\end{thm} 

\begin{proof}
Since $\mathcal T$ is a localizing subcategory of $Mod-\mathcal 
R$, the functor $\mathbb E_\tau: Mod-\mathcal R\longrightarrow Cl(\mathcal T)$ is left adjoint to the inclusion $Cl(\mathcal T)
\longrightarrow Mod-\mathcal R$. The ``unit'' of this adjunction gives a canonical morphism $V\longrightarrow \mathbb E_\tau(V)$ for each $V\in Mod-\mathcal R$. Taken together,
such maps induce a canonical morphism $i_\tau:\mathscr L\longrightarrow \mathbb E_\tau(\mathscr L)$ for each $\mathscr L\in FI_{\mathcal R}$.

\smallskip
 From the construction in \eqref{envel4}, it is clear that $i_\tau(S):\mathscr L(S)\longrightarrow \mathbb E_\tau(\mathscr L)(S)$ is a $\mathcal T$-isomorphism in $Mod-\mathcal R$
for each finite set $S$. In other words, $Ker(i_\tau(S))$, $Coker(i_\tau(S))\in \mathcal T$. Hence, $Ker(i_\tau)$, $Coker(i_\tau)\in \hat{\mathcal T^0}$ and it follows that
$i_\tau$ is a  $\hat{\mathcal T^0}$-isomorphism. From Lemma \ref{Lem4.5} and the definition in \eqref{envel4}, it is clear that $\mathbb E_\tau(\mathscr L)$ is $\hat{\mathcal T^0}$-closed. By Definition \ref{D4.1}, it follows that $ i_\tau :\mathscr L\longrightarrow \mathbb E_\tau(\mathscr L)$  is a $\hat{\mathcal T^0}$-envelope for $\mathscr L$. 

\smallskip
If we now suppose that $\mathscr L$ is $\hat{\mathcal T^0}$-closed and  $\hat{\mathcal T^0}$ is a hereditary torsion class, it follows from the uniqueness of the $\hat{\mathcal T^0}$-envelope that $i_\tau:\mathscr L\longrightarrow \mathbb E_\tau(\mathscr L)$ is an isomorphism. In particular, it follows that $\mathscr L(S)\cong \mathbb E_\tau(\mathscr L(S))$ is $\mathcal T$-closed in $Mod-\mathcal R$ for each
finite set $S$. This proves the ``only if'' part of the result. The ``if part'' is clear from Lemma \ref{Lem4.5}.

\end{proof}

From the definition of the subcategories $\{\hat{\mathcal T^a}\}_{a\geq 0}$, it  is clear that we have a descending filtration
\begin{equation}
Cl(\hat{\mathcal T^0})\supseteq Cl(\hat{\mathcal T^1})\supseteq \dots \qquad \dots \supseteq Cl(\hat{\mathcal T^a})\supseteq Cl(\hat{\mathcal T}^{a+1})\supseteq \dots 
\end{equation} In order to obtain functors going in the other direction, we will need to use the right adjoint of the shift functor $\mathbb S:=\mathbb S^1$. 

\begin{lem}\label{L4.7} For each $a\geq 0$, the functor $\mathbb S^a:FI_{\mathcal R}\longrightarrow FI_{\mathcal R}$ has a right adjoint 
$\mathbb T^a:FI_{\mathcal R}\longrightarrow FI_{\mathcal R}$.
\end{lem}

\begin{proof} Since $FI_{\mathcal R}$ is a Grothendieck category and $\mathbb S=\mathbb S^1:FI_{\mathcal R}\longrightarrow FI_{\mathcal R}$ preserves colimits,
it follows (see, for instance, \cite[Theorem 8.3.27]{KaSc}) that it must have a right adjoint $\mathbb T=\mathbb T^1:FI_{\mathcal R}\longrightarrow FI_{\mathcal R}$. Then, for each $a\geq 0$, $\mathbb T^a$ is a right
adjoint of $\mathbb S^a$. 

\end{proof}

\begin{lem}\label{Lem4.8} Let $k\geq 0$ and let $u:\mathscr A\longrightarrow \mathscr B$ be a $\hat{\mathcal T}^k$-isomorphism.  Then, for each
$0\leq a\leq k$, the induced morphism $\mathbb S^a(u):\mathbb S^a(\mathscr A)\longrightarrow \mathbb S^a(\mathscr B)$ is a $\hat{\mathcal T}^{k-a}$-isomorphism.
\end{lem}

\begin{proof}
For any finite set $S$, we know that 
\begin{equation}\label{eq4.12}
\begin{array}{c} Ker(\mathbb S^a(u))(S)=\mathbb S^a(Ker(u))(S)=Ker(u)(S\sqcup [-a])\\ Coker(\mathbb S^a(u))(S)=\mathbb S^a(Coker(u))(S)=Coker(u)(S\sqcup [-a])
\end{array}
\end{equation} Since $u:\mathscr A\longrightarrow \mathscr B$ is a $\hat{\mathcal T}^k$-isomorphism, it is clear from \eqref{eq4.12} that 
when $|S|\geq k-a$, both $Ker(\mathbb S^a(u))(S)$, $Coker(\mathbb S^a(u))(S)\in \mathcal T$. The result follows. 
\end{proof}

Before we proceed further, we record here the following observation about the functor $\mathbb T$. 

\begin{thm}\label{P4.85}Let $a$, $d\geq 0$ and $r\in \mathcal R$. Then, for any $\mathscr V\in FI_{\mathcal R}$, $\mathscr V(d)(r)$ is a direct summand of $\mathbb T^a(\mathscr V)(d)(r)$. 

\end{thm}

\begin{proof} From Lemma \ref{L2.3}, we know that $\mathbb T^a(\mathscr V)(d)(r)=FI_{\mathcal R}(_d\mathscr M_r,\mathbb T^a(\mathscr V))$ for any $d\geq 0$ and $r\in \mathcal R$. Using the adjoint pair $(\mathbb S^a,\mathbb T^a)$ and Corollary \ref{C3.15}, we obtain
\begin{equation*} \mathbb T^a(\mathscr V)(d)(r)=FI_{\mathcal R}(_d\mathscr M_r,\mathbb T^a(\mathscr V))=FI_{\mathcal R}(\mathbb S^a(_d\mathscr M_r),\mathscr V)=FI_{\mathcal R}({_{d}\mathscr M_{r}}\oplus {_{d}\mathscr N_{r}},\mathscr V)=\mathscr V(d)(r)\oplus FI_{\mathcal R}( {_{d}\mathscr N_{r}},\mathscr V)
\end{equation*}

\end{proof}

\begin{thm}\label{P4.9}
For any $a$, $b\geq 0$, the right adjoint $\mathbb T^a:FI_{\mathcal R}\longrightarrow FI_{\mathcal R}$ restricts to a functor 
$\mathbb T^a:Cl(\hat{\mathcal T}^b)\longrightarrow Cl(\hat{\mathcal T}^{a+b})$.
\end{thm}

\begin{proof}
We consider some $\mathscr L\in Cl(\hat{\mathcal T}^b)$ and $u:\mathscr A\longrightarrow \mathscr B$ in $FI_{\mathcal R}$ that is a $\hat{\mathcal T}^{a+b}$-isomorphism. We consider the commutative diagram:
\begin{equation}\label{eq4.13}
 \begin{CD}
FI_{\mathcal R}(\mathscr B,\mathbb T^a\mathscr L) @>FI_{\mathcal R}(u,\mathbb T^a\mathscr L)>>  FI_{\mathcal R}(\mathscr A,\mathbb T^a\mathscr L)\\
 @V\cong VV @VV\cong V\\
FI_{\mathcal R}(\mathbb S^a(\mathscr B),\mathscr L) @>FI_{\mathcal R}(\mathbb S^a(u),\mathscr L)>\cong > FI_{\mathcal R}(\mathbb S^a(\mathscr A),\mathscr L)\\
 \end{CD}
\end{equation} Here, it follows from Lemma \ref{Lem4.8} that the lower horizontal arrow is an isomorphism. This proves the result. 
\end{proof}

We can now give functors that explicitly construct objects in $Cl(\hat{\mathcal T}^k)$.

\begin{thm}\label{P4.10} Let $\tau=(\mathcal T,\mathcal F)$ be a hereditary torsion theory on $Mod-\mathcal R$. Then, we have functors
\begin{equation}\label{eq4.14} 
\mathbb L^k_\tau:=\mathbb T^k\circ \mathbb S^k\circ \mathbb E_\tau: FI_{\mathcal R}\longrightarrow Cl(\hat{\mathcal T}^k)\qquad\forall\textrm{ }k\geq 0
\end{equation} Additionally, there are  canonical morphisms of functors $l^k_\tau: Id\longrightarrow \mathbb L^k_\tau$ such that $l^{k+1}_\tau=
c^k_\tau\circ l^k_\tau$, where $c^k_\tau: \mathbb L^k_\tau=\mathbb T^k\circ \mathbb S^k\circ \mathbb E_\tau\longrightarrow \mathbb L^{k+1}_\tau=
\mathbb T^k\circ (\mathbb T\circ \mathbb S)\circ \mathbb S^k\circ \mathbb E_\tau$  is induced by the counit corresponding
to the adjoint pair $(\mathbb S,\mathbb T)$. 

\end{thm}

\begin{proof}
For any $\mathscr V\in FI_{\mathcal R}$, it is clear from Proposition \ref{P4.6} that  $\mathbb E_\tau(\mathscr V)\in Cl(\hat{\mathcal T}^0)$. From \eqref{envel5}, it is now clear that 
$\mathbb S^k(\mathbb E_\tau(\mathscr V))=\mathbb E_\tau(\mathbb S^k(\mathscr V))\in Cl(\hat{\mathcal T}^0)$. It now follows from Proposition \ref{P4.9} that $\mathbb L^k_\tau(\mathscr V)=\mathbb T^k\circ \mathbb S^k\circ \mathbb E_\tau(\mathscr V)\in Cl(\hat{\mathcal T}^k)$. We recall that we have a canonical morphism $i_\tau(\mathscr V'):\mathscr V'
\longrightarrow \mathbb E_\tau(\mathscr V')$ for each $\mathscr V'\in FI_{\mathbb R}$. Using the  adjunctions $(\mathbb S,\mathbb T)$, $(\mathbb S^k,\mathbb T^k)$ 
and $(\mathbb S^{k+1},\mathbb T^{k+1})$, we now have a commutative diagram
\begin{equation*}
\begin{tikzpicture}
      \matrix[matrix of math nodes,column sep=5pc,row sep=3em]
      {
                         |(A)| \mathscr V    &          \\
     |(C)|  \mathbb L^k_\tau(\mathscr V)  =\mathbb T^k\mathbb S^k\mathbb E_\tau(\mathscr V)  & |(D)|   \mathbb T^k(\mathbb T\circ \mathbb S)\mathbb S^k\mathbb E_\tau(\mathscr V) =
      \mathbb L^{k+1}_\tau(\mathscr V)    \\
      };
      \begin{scope}[->]
             \draw (A)  to node {$l^k_\tau(\mathscr V)\qquad\qquad$} (C);
         \draw (C) edge node [below] {$c^k_\tau(\mathscr V)$}  (D);
         \draw (A)  to node {$\qquad\qquad\qquad l^{k+1}_\tau(\mathscr V)$} (D);
      \end{scope}
   \end{tikzpicture}
   \end{equation*} The result follows. 
\end{proof}

We are now ready to construct a functor that gives objects that are closed with respect to $\hat{\mathcal T}$.

\begin{Thm}\label{P4.11}
Let $Mod-\mathcal R$ be locally noetherian. Let $\tau=(\mathcal T,\mathcal F)$ be a hereditary torsion theory on $Mod-\mathcal R$. Then, we have a functor $\mathbb L_\tau:FI_{\mathcal R}\longrightarrow Cl(\hat{\mathcal T})$ and a canonical morphism
\begin{equation}
l_\tau (\mathscr V): \mathscr V\longrightarrow \mathbb L_\tau(\mathscr V):=\underset{k\geq 0}{\varinjlim}\textrm{ }\mathbb L_\tau^k(\mathscr V)
\end{equation} for each $\mathscr V\in FI_{\mathcal R}$.
\end{Thm}
\begin{proof}
For $\mathscr V\in FI_{\mathcal R}$, it follows from Proposition \ref{P4.10} that the morphisms $l^k_\tau(\mathscr V):\mathscr V\longrightarrow \mathbb L^k_\tau(\mathscr V)$ combine to give a morphism
$l_\tau (\mathscr V): \mathscr V\longrightarrow \mathbb L_\tau(\mathscr V)=\underset{k\geq 0}{\varinjlim}\textrm{ }\mathbb L_\tau^k(\mathscr V)$. We need to check
that $\mathbb L_\tau(\mathscr V)$ is $\hat{\mathcal T}$-closed. For this, we will show that for any $\hat{\mathcal T}$-isomorphism $u:\mathscr A\longrightarrow\mathscr B$, the induced morphism $FI_{\mathcal R}(u, \mathbb L_\tau(\mathscr V)):FI_{\mathcal R}(\mathscr B, \mathbb L_\tau(\mathscr V))\longrightarrow FI_{\mathcal R}(\mathscr A, \mathbb L_\tau(\mathscr V))$ is an isomorphism.

\smallskip Using Proposition \ref{P4.4}, we may restrict ourselves to the case where $\mathscr A$, $\mathscr B$ are finitely generated. Since $FI_{\mathcal R}$ is a locally noetherian category, 
it follows that $\mathscr A$, $\mathscr B$ are also finitely presented, i.e., the functors $FI_{\mathcal R}(\mathscr A,\_\_)$, $FI_{\mathcal R}(\mathscr B,\_\_)$ preserve filtered colimits. We now consider the morphism
\begin{equation*}
FI_{\mathcal R}(u, \mathbb L_\tau(\mathscr V)):FI_{\mathcal R}(\mathscr B, \mathbb L_\tau(\mathscr V))=\underset{k\geq 0}{\varinjlim}\textrm{ }FI_{\mathcal R}(\mathscr B, \mathbb L_\tau^k(\mathscr V))\longrightarrow \underset{k\geq 0}{\varinjlim}\textrm{ }FI_{\mathcal R}(\mathscr A, \mathbb L_\tau^k(\mathscr V))=FI_{\mathcal R}(\mathscr A, \mathbb L_\tau(\mathscr V))
\end{equation*} Since $u:\mathscr A\longrightarrow\mathscr B$ is a $\hat{\mathcal T}$-isomorphism, we can choose $N$ large enough so that $Ker(u)(S)$, $Coker(u)(S)\in 
\mathcal T$ for finite sets $S$ of cardinality $\geq N$. Hence, $u:\mathscr A\longrightarrow\mathscr B$ is a $\hat{\mathcal T}^k$-isomorphism for each $k\geq N$. Since
$\mathbb L^k_\tau(\mathscr V)$ is $\hat{\mathcal T}^k$-closed, it follows that $FI_{\mathcal R}(\mathscr B, \mathbb L_\tau^k(\mathscr V))\longrightarrow FI_{\mathcal R}(\mathscr A, \mathbb L_\tau^k(\mathscr V)$ is an isomorphism for each $k\geq N$. The result is now clear. 
\end{proof}

\section{The torsion class $\widetilde{\mathcal T}$ and its closed objects}

\smallskip
We continue with $Mod-\mathcal R$ being a locally noetherian category and $\tau=(\mathcal T,\mathcal F)$ being a hereditary torsion theory on $Mod-\mathcal R$. In Section 3, we used this torsion theory 
to construct a torsion class $\overline{\mathcal T}$ in the category $FI^{fg}_{\mathcal R}$ of finitely generated modules.  In Section 4, we considered the Serre subcategory
$\hat{\mathcal T}\subseteq FI_{\mathcal R}$ which was constructed so that $\hat{\mathcal T}\cap FI^{fg}_{\mathcal R}=\overline{\mathcal T}$. 

\smallskip
Additionally, in Section 3, we had also used the torsion theory $\tau=(\mathcal T,\mathcal F)$ to construct a torsion class $\overline{\mathcal T}^{sfg}$ in the category $FI^{sfg}_{\mathcal R}$ of shift finitely generated modules. As such, in this section, we will define a full subcategory $\widetilde{\mathcal T}\subseteq FI_{\mathcal R}$ such that $\widetilde{\mathcal T}\cap FI^{sfg}_{\mathcal R}=\overline{\mathcal T}^{sfg}$. For this, we define:
\begin{equation}\label{eq5.1nw}
Ob(\widetilde{\mathcal T}):=\{\mbox{$\mathscr V\in Ob(FI_{\mathcal R})$ $\vert$ Every finitely generated $\mathscr W\subseteq \mathscr V$ lies in $\overline{\mathcal T}$
 }\}
\end{equation}
The subcategory $\hat{\mathcal T}$ considered in Section 4 was a Serre subcategory. Hence, we would expect that its counterpart $\widetilde{\mathcal T}$ defined in \eqref{eq5.1nw}  is also  a Serre subcategory. We will now show that $\widetilde{\mathcal T}$ satisfies an even stronger property,
i.e., it is a hereditary torsion class. 

\smallskip
\begin{lem}\label{L5.1nw} The subcategory $\widetilde{\mathcal T}$ is closed under extensions. In other words, suppose that we have a short exact sequence
\begin{equation*}
0\longrightarrow \mathscr V'\longrightarrow \mathscr V\longrightarrow \mathscr V''\longrightarrow 0
\end{equation*} in $FI_{\mathcal R}$ with $\mathscr V'$, $\mathscr V''\in \widetilde{\mathcal T}$. Then, $\mathscr V\in \widetilde{\mathcal T}$. 
\end{lem}

\begin{proof} We consider a finitely generated subobject $\mathscr W\subseteq \mathscr V$. This gives two short exact sequences fitting into the commutative diagram
\begin{equation}\label{5.2vo}
\begin{CD}
0 @>>> \mathscr W\cap \mathscr V' @>>> \mathscr W@>>> \mathscr W/(\mathscr W\cap \mathscr V')=(\mathscr W+\mathscr V')/\mathscr V'@>>> 0\\
@. @VVV @VVV @VVV @. \\
0 @>>> \mathscr V' @>>> \mathscr V @>>> \mathscr V'' =\mathscr V/\mathscr V'@>>> 0\\
\end{CD}
\end{equation} Since $\mathscr W$ is finitely generated, so is the subobject $\mathscr W\cap \mathscr V' $ and the quotient $\mathscr W/(\mathscr W\cap \mathscr V')$. The vertical
maps in \eqref{5.2vo} are all monomorphisms. Since $\mathscr V'$, $\mathscr V''\in \widetilde{\mathcal T}$, we can find $N$ large enough so that $(\mathscr W\cap \mathscr V' )_n$, $
(\mathscr W/(\mathscr W\cap \mathscr V'))_n\in \mathcal T$ for $n\geq N$. Then, $\mathscr W_n\in \mathcal T$ for all $n\geq N$.

\end{proof}

\begin{lem}\label{L5.2nw}  The subcategory $\widetilde{\mathcal T}$ contains all coproducts. 
\end{lem}

\begin{proof}
Let $\{\mathscr V_i\}_{i\in I}$ be a family of objects in $\widetilde{\mathcal T}$. Using Lemma \ref{L5.1nw}, we know that every finite direct sum of objects
from $\{\mathscr V_i\}_{i\in I}$ lies in $\widetilde{\mathcal T}$. We note that $\underset{i\in I}{\bigoplus}\textrm{ }\mathscr V_i$ is equal to the filtered colimit of $
\underset{j\in J}{\bigoplus}\textrm{ }\mathscr V_j$ taken over all finite subsets $J\subseteq I$. Then, if $\mathscr W\subseteq \underset{i\in I}{\bigoplus}\textrm{ }\mathscr V_i$
 is a finitely generated object, we can find some finite subset $J\subseteq I$ such that $\mathscr W\subseteq \underset{j\in J}{\bigoplus}\textrm{ }\mathscr V_j$. It follows that
 $\mathscr W_n\in \mathcal T$ for $n\gg 0$. This proves the result. 
\end{proof}

\begin{lem}\label{L5.3nw} The subcategory $\widetilde{\mathcal T}$  is closed under quotients.

\end{lem}

\begin{proof}
We consider an epimorphism $f:\mathscr V\longrightarrow \mathscr W$ in $FI_{\mathcal R}$ with $\mathscr V\in \widetilde{\mathcal T}$. We consider a finitely generated
subobject $\mathscr W'\subseteq \mathscr W$.  The finitely generated subobjects of $f^{-1}(
\mathscr W')\subseteq \mathscr V$ form a filtered system and hence their images in $\mathscr W'$ form a filtered system of subobjects whose union is $\mathscr W'$. As such, we can
find some finitely generated subobject $\mathscr V'\subseteq f^{-1}(\mathscr W')\subseteq \mathscr V$ such that $f|\mathscr V':\mathscr V'\longrightarrow \mathscr W'$ is an epimorphism. Since $\mathscr V\in \widetilde{\mathcal T}$, we know that $\mathscr V'_n\in \mathcal T$ for $n\gg 0$. Then, $\mathscr W'_n\in \mathcal T$ for $n\gg 0$. This proves the result.
\end{proof}

\begin{thm}\label{P5.4nw}  Let  $Mod-\mathcal R$ be a locally noetherian category and $\tau=(\mathcal T,\mathcal F)$ a hereditary torsion theory on $Mod-\mathcal R$. Then, the full
subcategory $\widetilde{\mathcal T}\subseteq FI_{\mathcal R}$ defined by setting
\begin{equation}
Ob(\widetilde{\mathcal T}):=\{\mbox{$\mathscr V\in Ob( FI_{\mathcal R})$ $\vert$ Every f.g. $\mathscr W\subseteq \mathscr V$ satisfies $\mathscr W_n\in \mathcal T$
for $n\gg 0$ }\}
\end{equation}
is a hereditary torsion class.
\end{thm}

\begin{proof}
From the definition, it is clear that $\widetilde{\mathcal T}$ is closed under subobjects. From Lemma \ref{L5.1nw}, Lemma \ref{L5.2nw} and Lemma \ref{L5.3nw} we know respectively that
$\widetilde{\mathcal T}$ is closed under extensions, coproducts and quotients. Since $FI_{\mathcal R}$ is a Grothendieck category, it now follows from 
\cite[I.1]{BeRe} that $\widetilde{\mathcal T}$ is a hereditary torsion class.
\end{proof}

Now that we know $\widetilde{\mathcal T}$ is a hereditary torsion class, our aim is to describe closed objects with respect to $\widetilde{\mathcal T}$ as well as a functor
from $FI_{\mathcal R}$ to $Cl(\widetilde{\mathcal T})$. This will be done in several steps. 

\smallskip
For any $\mathscr V\in FI_{\mathcal R}$ and any $n\geq 0$, we now denote by $fg^\tau_n(\mathscr V)$ the collection of all finitely generated subobjects $\mathscr W'
\subseteq \mathscr V$ such that $\mathscr W'_m\in \mathcal T$ for every $m\geq n$. Clearly, $fg^\tau_n(\mathscr V)\subseteq \hat{\mathcal T}^n$. 

\begin{thm}\label{P5.5}  Let  $Mod-\mathcal R$ be a locally noetherian category and $\tau=(\mathcal T,\mathcal F)$ a hereditary torsion theory on $Mod-\mathcal R$. Every
$\mathscr V\in \widetilde{\mathcal T}$ is equipped with an increasing filtration 
\begin{equation}\label{5.4ner}
F^0\mathscr V\subseteq F^1\mathscr V \subseteq .... \subseteq \mathscr V
\end{equation}
with $F^n\mathscr V\in \hat{\mathcal T}^n$ for each $n\geq 0$. 

\end{thm}

\begin{proof} We let $F^n\mathscr V$ be the sum of all $\mathscr W\in fg^\tau_n(\mathscr V)$. Since $\mathscr V\in \widetilde{\mathcal T}$, every finitely generated subobject
$\mathscr W\subseteq \mathscr V$ lies in $fg_n^\tau(\mathscr V)$ for some $n\geq 0$. Since $\mathscr V$ is the sum of its finitely generated subobjects, it follows that we have a filtration as in \eqref{5.4ner} whose union is $\mathscr V$. 

\smallskip
It remains to show that $F^n\mathscr V\in \hat{\mathcal T}^n$ for each $n\geq 0$.  We note that any finite sum $\underset{j\in J}{\sum}\mathscr W_j$ of objects in $ fg_n^\tau
(\mathscr V)$ is a quotient of the direct sum $\underset{j\in J}{\bigoplus}\mathscr W_j$ and hence lies in $fg^\tau_n(\mathscr V)$. We also observe that $F^n\mathscr V$ is the 
colimit of $\underset{j\in J}{\sum}\mathscr W_j$ as $J$ varies over all finite  collections of objects in $fg^\tau_n(\mathscr V)$.  Since $\mathcal T$ is closed under colimits (being a torsion class), the result follows.

\end{proof}

\begin{lem}\label{L5.6f} Let $u:\mathscr A\longrightarrow \mathscr B$ be an epimorphism in $FI_{\mathcal R}$ with $Ker(u)=\mathscr K\in \widetilde{\mathcal T}$. Let $\mathscr L
\in Cl(\hat{\mathcal T})$. Then, the induced morphism
$FI_{\mathcal R}(u,\mathscr L):FI_{\mathcal R}(\mathscr B,\mathscr L)\longrightarrow FI_{\mathcal R}(\mathscr A,\mathscr L)$ is an isomorphism.

\end{lem}

\begin{proof} Using Proposition \ref{P5.5},we can consider a filtration $\mathscr K^0\subseteq \mathscr K^1\subseteq ....$ on $\mathscr K\in \widetilde{\mathcal T}$ with each $\mathscr K^i\in \hat{\mathcal T}^i$. Since $\hat{\mathcal T}^i\subseteq \hat{\mathcal T}$, we know that $\mathscr L
\in Cl(\hat{\mathcal T})$ lies in $Cl(\hat{\mathcal T}^i)$ for each $i$. It follows therefore that we have an isomorphism $FI_{\mathcal R}(\mathscr A/
\mathscr K^i,\mathscr L)\longrightarrow FI_{\mathcal R}(\mathscr A,\mathscr L)$ for each $i$. We know that $\mathscr B=\mathscr A/\mathscr K$. Taking limits, we therefore obtain an isomorphism
\begin{equation*}
FI_{\mathcal R}(u,\mathscr L):FI_{\mathcal R}(\mathscr B,\mathscr L)=FI_{\mathcal R}(\underset{i\geq 0}{\varinjlim}\textrm{ }\mathscr A/
\mathscr K^i,\mathscr L)=\underset{i\geq 0}{\varprojlim}\textrm{ }FI_{\mathcal R}(\mathscr A/
\mathscr K^i,\mathscr L)\overset{\cong}{\longrightarrow} FI_{\mathcal R}(\mathscr A,\mathscr L)
\end{equation*}

\end{proof}

\begin{lem}\label{L5.7f} Let $u:\mathscr A\longrightarrow \mathscr B$ be a  monomorphism in $FI_{\mathcal R}$ with $Coker(u)=\mathscr C\in \widetilde{\mathcal T}$. Let $\mathscr L
\in Cl(\hat{\mathcal T})$. Then, the induced morphism
$FI_{\mathcal R}(u,\mathscr L):FI_{\mathcal R}(\mathscr B,\mathscr L)\longrightarrow FI_{\mathcal R}(\mathscr A,\mathscr L)$ is an isomorphism.

\end{lem}

\begin{proof} Using Proposition \ref{P5.5},we can consider a filtration $\mathscr C^0\subseteq \mathscr C^1\subseteq ....$ on $\mathscr C\in \widetilde{\mathcal T}$ with each $\mathscr C^i\in \hat{\mathcal T}^i$. This corresponds to a filtration $\mathscr B^0\subseteq \mathscr B^1\subseteq ....$ on $\mathscr B$ such that $\mathscr B^i/\mathscr A=\mathscr C^i$. 

\smallskip Since $\hat{\mathcal T}^i\subseteq \hat{\mathcal T}$, we know that $\mathscr L
\in Cl(\hat{\mathcal T})$ lies in $Cl(\hat{\mathcal T}^i)$ for each $i$. It follows therefore that we have an isomorphism $FI_{\mathcal R}(\mathscr B^i
,\mathscr L)\longrightarrow FI_{\mathcal R}(\mathscr A,\mathscr L)$ for each $i$. Taking limits, we therefore obtain an isomorphism
\begin{equation*}
FI_{\mathcal R}(u,\mathscr L):FI_{\mathcal R}(\mathscr B,\mathscr L)=FI_{\mathcal R}(\underset{i\geq 0}{\varinjlim}\textrm{ }
\mathscr B^i,\mathscr L)=\underset{i\geq 0}{\varprojlim}\textrm{ }FI_{\mathcal R}(
\mathscr B^i,\mathscr L)\overset{\cong}{\longrightarrow} FI_{\mathcal R}(\mathscr A,\mathscr L)
\end{equation*}

\end{proof} 

\begin{Thm}\label{Tmh5.8} Let  $Mod-\mathcal R$ be a locally noetherian category and $\tau=(\mathcal T,\mathcal F)$ a hereditary torsion theory on $Mod-\mathcal R$. Then,
$Cl(\hat{\mathcal T})=Cl(\widetilde{\mathcal T})$. In particular, we have a functor $\mathbb L_\tau:FI_{\mathcal R}\longrightarrow Cl(\hat{\mathcal T})=Cl(\widetilde{\mathcal T})$. 

\end{Thm}

\begin{proof}
From the definitions in \eqref{4.3ne} and \eqref{eq5.1nw}, we know that $\hat{\mathcal T}\subseteq \widetilde{\mathcal T}$, whence it follows that $Cl(\hat{\mathcal T})\supseteq Cl(\widetilde{\mathcal T})$. We now consider a morphism $u:\mathscr A\longrightarrow \mathscr B$ in $FI_{\mathcal R}$ which is $\widetilde{\mathcal T}$-closed. Then, $u$ may be expressed as the composition
\begin{equation}
u: \mathscr A\xrightarrow{u'}\mathscr A/Ker(u) \xrightarrow{u''} \mathscr B
\end{equation} where $u'$ is an epimorphism with $Ker(u')\in \widetilde{\mathcal T}$ and $u''$ is a monomorphism with $Coker(u'')\in \widetilde{\mathcal T}$.  Let $\mathscr L
\in Cl(\hat{\mathcal T})$. From Lemma \ref{L5.6f} and Lemma \ref{L5.7f}, it follows that $FI_{\mathcal R}(u',\mathscr L)$ and $FI_{\mathcal R}(u'',\mathscr L)$ are both
isomorphisms. Hence, $FI_{\mathcal R}(u,\mathscr L)$ is an isomorphism. Hence, $\mathscr L\in Cl(\widetilde{\mathcal T})$. This proves the result.
\end{proof}

\section{The functors $H_a$ and properties of finitely and shift finitely generated modules}

We return to the general case, i.e., $\mathcal R$ is a small preadditive category, but $Mod-\mathcal R$ is not necessarily noetherian. Fix  $a\geq 0$. Let $\mathscr V
\in FI_{\mathcal R}$. In a manner similar to 
\cite[$\S$ 2]{Four1}, we define the functor 
\begin{equation}\label{B-a}
\mathbb B^{-a}:FI_{\mathcal R}\longrightarrow FI_{\mathcal R}\qquad \mathbb B^{-a}(\mathscr V)(S):=\underset{\phi:[a]\hookrightarrow S}{\bigoplus}\textrm{ }\mathscr V(S,\phi)
=\underset{\phi:[a]\hookrightarrow S}{\bigoplus}\textrm{ }\mathscr V(S-\phi[a])
\end{equation} It is clear from \eqref{B-a} that $\mathbb B^{-a}(\mathscr V)\in FI_{\mathcal R}$. For each $a\geq 1$, we consider the set $\{s_i:[a-1]\longrightarrow [a]\}_{1\leq i\leq a}$ of standard  order-preserving injections, where the image of $s_i$ misses $i$. Then, for any $\phi:[a]\longrightarrow S$, we have $ S-\phi\{a\}\subsetneq S-\phi\circ s_i\{a-1\}$ which induces a morphism $d_i(S,\phi):\mathscr V( S-\phi\{a\})\longrightarrow \mathscr V( S-\phi\circ s_i\{a-1\})$. Taking the alternating sum $\sum_{i=1}^a (-1)^id_i$ of these maps in the usual manner, we obtain a complex
\begin{equation}\label{Bstar}
\mathbb B^{-\ast}(\mathscr V) : \qquad \dots \longrightarrow \mathbb B^{-a}(\mathscr V)\longrightarrow \mathbb B^{-(a-1)}(\mathscr V) \longrightarrow \dots 
\longrightarrow \mathbb B^{-1}(\mathscr V) \longrightarrow \mathbb B^0(\mathscr V)=\mathscr V\longrightarrow 0
\end{equation} Let $S_a$ be the permutation group on $a$ objects and consider the group ring $\mathbb Z[S_a]$. We consider the small preadditive category 
$\mathcal R[S_a]$ defined by setting $Ob(\mathcal R)=Ob(\mathcal R[S_a])$ and 
\begin{equation}\label{eq5.3}
\mathcal R[S_a](r,r'):=\mathcal R(r,r')\otimes_{\mathbb Z}\mathbb Z[S_a]
\end{equation} The composition in $\mathcal R[S_a]$ is the usual composition in $\mathcal R$ extended by the multiplication in $\mathbb Z[S_a]$. Given a morphism $f\cdot \sigma
\in \mathcal R[S_a](r,r')$, i.e.,  $f\in \mathcal R(r,r')$ and $\sigma\in S_a$, we notice that we have a map 
\begin{equation}\label{eq5.4}
\mathscr V(S,\phi)(r')
=\mathscr V(S-\phi[a])(r')\xrightarrow{\mathscr V(S-\phi[a])(f)} \mathscr V(S-\phi[a])(r)=\mathscr V(S-\phi\circ \sigma[a])(r)=\mathscr V(S,\phi\circ \sigma)(r)
\end{equation} for each $\phi:[a]\longrightarrow S$ in $FI$. Using the maps in \eqref{eq5.4}, it may be easily verified that
$\mathbb B^{-a}$ may be treated as a functor $\mathbb B^{-a}:FI_{\mathcal R}\longrightarrow FI_{\mathcal R[S_a]}$. On the other hand, the canonical $\mathcal R\text{-}\mathcal R$-bimodule given by morphism spaces in $\mathcal R$ may be extended to a left $\mathcal R[S_a]$ right $\mathcal R$-module:
\begin{equation}
\begin{array}{c}
H_{\mathcal R}^\varepsilon:\mathcal R^{op}\otimes \mathcal R[S_a]\longrightarrow Ab \qquad (r',r)\mapsto \mathcal R(r',r)\\
H_{\mathcal R}^\varepsilon (f_1,f_2\cdot \sigma):H_{\mathcal R}^\varepsilon(r',r)\longrightarrow H_{\mathcal R}^\varepsilon(r''',r'') \qquad f\mapsto (-1)^{sgn(\sigma)}f_2\circ f\circ f_1\\
\end{array}
\end{equation} Here $sgn(\sigma)$ is the sign of the permutation in $\mathbb Z_2$. This allows us to define a functor 
\begin{equation}\label{btilde-a}
\widetilde{\mathbb B}^{-a}:FI_{\mathcal R}\longrightarrow FI_{\mathcal R}\qquad \widetilde{\mathbb B}^{-a}(\mathscr V)(S):=\mathbb B^{-a}(\mathscr V)(S)\otimes_{\mathcal R[S_a]}H_{\mathcal R}^\varepsilon
\end{equation} We observe that $ \widetilde{\mathbb B}^{-a}(\mathscr V)(S)\in Mod-\mathcal R$ is a direct sum of all $\mathscr V(T)$ as $T$ varies over all the distinct subsets of
$S$ such that $|T|=|S|-a$.  It may be verified by direct computation that the complex in \eqref{Bstar} descends to a complex
\begin{equation}\label{Btildestar}
\widetilde{\mathbb B}^{-\ast}(\mathscr V) : \qquad \dots \longrightarrow \widetilde{\mathbb B}^{-a}(\mathscr V)\longrightarrow \widetilde{\mathbb B}^{-(a-1)}(\mathscr V) \longrightarrow \dots 
\longrightarrow \widetilde{\mathbb B}^{-1}(\mathscr V) \longrightarrow \widetilde{\mathbb B}^0(\mathscr V)=\mathscr V\longrightarrow 0
\end{equation} For $\mathscr V\in FI_{\mathcal R}$, we set $H_{a}(\mathscr V):=H^{-a}(\widetilde{\mathbb B}^{-\ast}(\mathscr V) )\in FI_{\mathcal R}$. In particular, it is easy to observe that $\mathbb B^{-a}(_{d}\mathscr M_r)={_{a+d}\mathscr M_r}$. Since $\mathbb B^{-a}$ is exact, it follows that for any $\mathscr V\in FI^{fg}_{\mathcal R}$, 
the object $\mathbb B^{-a}(\mathscr V)$ is finitely generated. Further, since $\widetilde{\mathbb B}^{-a}(\mathscr V)$ is a quotient of $\mathbb B^{-a}(\mathscr V)$, 
it follows that $\widetilde{\mathbb B}^{-a}(\mathscr V)\in FI_{\mathcal R}^{fg}$.

\begin{thm}\label{P5.2} Let $\mathscr V\in FI_{\mathcal R}$. Fix a finite set $S$ and consider the colimit $\underset{T\subsetneq S}{colim}\textrm{ }\mathscr V(T)$ taken over
all proper subsets of $S$ ordered by inclusion. Then, we have
\begin{equation}\label{eq5.9}
H^{0}(\widetilde{\mathbb B}^{-\ast}(\mathscr V) )(S)= Coker\left(\underset{T\subsetneq S}{colim}\textrm{ }\mathscr V(T)\longrightarrow \mathscr V(S)\right)=H_0(\mathscr V)(S)
\end{equation}
\begin{equation}\label{eq5.10}
H^{-1}(\widetilde{\mathbb B}^{-\ast}(\mathscr V) )(S)= Ker\left(\underset{T\subsetneq S}{colim}\textrm{ }\mathscr V(T)\longrightarrow \mathscr V(S)\right)=H_1(\mathscr V)(S)
\end{equation}

\end{thm}

\begin{proof}
For any injection $\phi: T'\hookrightarrow S$ with $|T'|<|S|$, it is obvious that $\phi$ factors through a proper subset of $S$. Comparing with the definition in \eqref{h0}, we see
that $H_0(\mathscr V)= Coker\left(\underset{T\subsetneq S}{colim}\textrm{ }\mathscr V(T)\longrightarrow \mathscr V(S)\right)$. From the discussion above, we know that  $ \widetilde{\mathbb B}^{-1}(\mathscr V)(S)\in Mod-\mathcal R$ is a direct sum of all $\mathscr V(T)$ as $T$ varies over all the distinct subsets of
$S$ such that $|T|=|S|-1$. Since the  inclusion of any proper subset of $S$ factors through a subset of size $|S|-1$, we also observe that $H^{0}(\widetilde{\mathbb B}^{-\ast}(\mathscr V) )(S)=Coker( \widetilde{\mathbb B}^{-1}(\mathscr V)(S)\longrightarrow \mathscr V(S))= Coker\left(\underset{T\subsetneq S}{colim}\textrm{ }\mathscr V(T)\longrightarrow \mathscr V(S)\right)$. This proves \eqref{eq5.9}.

\smallskip
To prove \eqref{eq5.10}, we proceed as follows: for each subset $T\subseteq S$ of cardinality $|S|-2$, there are exactly two subsets $T_1,T_2\subseteq S$ each of cardinality $|S|-1$ such that $T\subseteq T_1,T_2$. This induces 
maps $\mathscr V(T)\longrightarrow \mathscr V(T_1)$ and $\mathscr V(T)\longrightarrow \mathscr V(T_2)$.  We now observe that
\begin{equation}\label{eq5.11}
\underset{T\subsetneq S}{colim}\textrm{ }\mathscr V(T)=Coeq\left(\xymatrix{\underset{\mbox{\tiny $\begin{array}{c} T\subseteq  S \\ |T|=|S|-2
\end{array}$}}{\bigoplus}\mathscr V(T) \ar@<-.5ex>[r]_(0.4){} \ar@<.5ex>[r]^(0.4){}& \underset{\mbox{\tiny $\begin{array}{c} T\subseteq  S \\ |T|=|S|-1
\end{array}$}}{\bigoplus}\mathscr V(T) }  \right)
\end{equation} From the definition of the differential in the complex in \eqref{Btildestar}, it is clear that the expression in \eqref{eq5.11} is identical to
$Coker( \widetilde{\mathbb B}^{-2}(\mathscr V) \longrightarrow \widetilde{\mathbb B}^{-1}(\mathscr V))(S)$. It follows that
\begin{equation*}
H^{-1}(\widetilde{\mathbb B}^{-\ast}(\mathscr V) )(S)= Ker(Coker( \widetilde{\mathbb B}^{-2}(\mathscr V) \longrightarrow \widetilde{\mathbb B}^{-1}(\mathscr V))(S)
\longrightarrow\mathscr V(S))= Ker\left(\underset{T\subsetneq S}{colim}\textrm{ }\mathscr V(T)\longrightarrow \mathscr V(S)\right)
\end{equation*}
\end{proof}

\begin{thm}\label{P5.3} Let $\mathscr V\in FI_{\mathcal R}$. Then, for each $a\geq 0$, the canonical map $H_a(\mathscr V)\longrightarrow \mathbb S^1 H_a(\mathscr V)$ is zero. 

\end{thm} 

\begin{proof}
We consider the system of maps 
\begin{equation}
\{G^{-b}:\mathbb B^{-b}\mathscr V\longrightarrow \mathbb S^1\mathbb B^{-b-1}\mathscr V\}_{b\geq 0}
\end{equation} defined as follows: for a finite set $S$ and a map $\phi: [b]\longrightarrow S$, denote by $\bar\phi:[b+1]\longrightarrow S\sqcup [-1]$ the map given by
\begin{equation}
\bar\phi(i)=\left\{\begin{array}{ll}
\ast & \mbox{if $i=1$}\\
\phi(i-1) & \mbox{otherwise} \\
\end{array}\right.
\end{equation} where $[-1]$ has been chosen to be the single element set $\{\ast\}$.  Then, the identifications
\begin{equation}
 \mathscr V(S,\phi)=\mathscr V(S-\phi[b]) \overset{=}{\longrightarrow} \mathscr V(S\sqcup [-1]-\bar\phi[b+1])=\mathscr V(S\sqcup [-1],\bar\phi)
\end{equation} combine to determine the map $G^{-b}(S):\mathbb B^{-b}\mathscr V(S)\longrightarrow \mathbb B^{-b-1}\mathscr V(S\sqcup [-1])=\mathbb S^1\mathbb B^{-b-1}\mathscr V(S)$. As in the proof of \cite[Proposition 2.25]{Four1}, it may be verified that the maps $G^{-b}$ induce a homotopy equivalence between the zero map and the canonical
map $\widetilde{\mathbb B}^{-\ast}(\mathscr V) \longrightarrow \mathbb S^1\widetilde{\mathbb B}^{-\ast}(\mathscr V) $. 
\end{proof}

\begin{thm}\label{P5.4} Suppose that $Mod-\mathcal R$ is locally noetherian. Let $\mathscr V\in FI_{\mathcal R}$ be a finitely generated object. Then, for each $a\geq 0$, there exists
$N\geq 0$ such that $H_a(\mathscr V)_n=0$ for all $n\geq N$. 
\end{thm}

\begin{proof} We have explained before that if   $\mathscr V\in FI_{\mathcal R}$ is finitely generated, $\widetilde{\mathbb B}^{-a}(\mathscr V)$
is finitely generated. Since $FI_{\mathcal R}$ is locally noetherian, it follows that $H_a(\mathscr V)$ is also  finitely generated. We consider the trivial torsion theory $\tau_0$ on $Mod-\mathcal R$ whose torsion class is $0$. Using Proposition \ref{xxP3.2}, this induces a torsion class on $FI^{fg}_{\mathcal R}$ whose torsion class $\overline{\mathcal T}_0$ is given by
\begin{equation}
Ob(\overline{\mathcal T}_0):=\{\mbox{$\mathscr V\in Ob(FI_{\mathcal R}^{fg})$ $\vert$ $\mathscr V_n=0$ for $n\gg 0$}\}
\end{equation}  Using Theorem \ref{P3.8cf}, we know that the torsion subobject of  $H_a(\mathscr V)$ is given by 
\begin{equation}\label{ystors3}
\begin{array}{ll}
\overline{\mathcal T}_0(H_a(\mathscr V))(S)&=\underset{b\geq 0}{colim}\textrm{ }lim\left(
\begin{CD}H_a(\mathscr V )(S)@>\psi_b^{H_a(\mathscr V)}(S) >>{\mathbb S}^bH_a(\mathscr V )(S) @<<<\mathcal T_0( {\mathbb S}^bH_a(\mathscr V)(S))
\end{CD}\right)\\
&=\underset{b\geq 1}{colim}\textrm{ }lim\left(
\begin{CD}H_a(\mathscr V )(S)@>\psi_b^{H_a(\mathscr V)}(S) >>{\mathbb S}^bH_a(\mathscr V )(S) @<<<0
\end{CD}\right)\\
\end{array}
\end{equation} From Proposition \ref{P5.3} and the expression in \eqref{ystors3}, it now follows that $\overline{\mathcal T}_0(H_a(\mathscr V))(S)=H_a(\mathscr V)(S)$. Hence,
$H_a(\mathscr V)\in \overline{\mathcal T}_0$ and the result follows. 

\end{proof}

We now have  an analogue of \cite[Theorem C]{Four1}.

\begin{Thm}\label{T5.5}  Suppose that $Mod-\mathcal R$ is locally noetherian. Let $\mathscr V\in FI_{\mathcal R}$ be a finitely generated object.  Then, there exists $N\geq 0$ such that
\begin{equation}\label{5.17xe}
\underset{\tiny \begin{array}{c} T\subseteq S \\ |T|\leq N\\ \end{array}}{colim}\textrm{ }\mathscr V(T)=\mathscr V(S)
\end{equation} for each finite set $S$.

\end{Thm}

\begin{proof} Using Proposition \ref{P5.4}, we can choose $N\geq 1$ such that $H_0(\mathscr V)_n=H_1(\mathscr V)_n=0$ for all $n\geq N$. It is clear that
\eqref{5.17xe} holds for all $S$ such that $|S|\leq N$.  We consider a set $S$ with  $|S|>N$ and suppose that \eqref{5.17xe} holds for all finite sets $U$ of cardinality $<|S|$. We observe that
\begin{equation}\label{5.18xe}
\underset{\tiny \begin{array}{c} T\subseteq S \\ |T|\leq N\\ \end{array}}{colim}\textrm{ }\mathscr V(T)=\underset{U\subsetneq S}{colim}\textrm{ }\underset{\tiny \begin{array}{c} T\subseteq U \\ |T|\leq N\\ \end{array}}{colim}\textrm{ }\mathscr V(T)
\end{equation} Since each $U$ appearing in \eqref{5.18xe} has cardinality $<|S|$, we have
\begin{equation}
\underset{\tiny \begin{array}{c} T\subseteq S \\ |T|\leq N\\ \end{array}}{colim}\textrm{ }\mathscr V(T)=\underset{U\subsetneq S}{colim}\textrm{ }\mathscr V(U)
\end{equation} Finally since $|S|>N$, we know that $H_0(\mathscr V)(S)=H_1(\mathscr V)(S)=0$. The result is now clear from the expressions in Proposition \ref{P5.2}. 
\end{proof}

So far in this section, we have used the properties of $H_a(\mathscr V)$ for $\mathscr V$ finitely generated. We will now consider the objects $H_a(\mathscr V)$ when $\mathscr V$ is shift finitely generated. 

\begin{lem}\label{L5.6} Let $\mathscr V\in FI_{\mathcal R}$ be shift finitely generated. Then, for any $a\geq  0$, $\mathbb B^{-a}(\mathscr V)$ is also shift
finitely generated.
\end{lem}

\begin{proof}
Since $\mathscr V\in FI_{\mathcal R}^{sfg}$, we choose $d\geq 0$ such that $\mathbb S^d\mathscr V$ is finitely generated. We choose $e\geq a+d$. We will show that
$\mathbb S^{e}\mathbb B^{-a}(\mathscr V)$ is finitely generated. For any finite set $T$, we see that
\begin{equation}\label{eq5/20}
\begin{array}{l}
\mathbb S^{e}\mathbb B^{-a}(\mathscr V)(T)=\mathbb B^{-a}\mathscr V(T\sqcup [-e])\\
=\underset{\phi:[a]\rightarrow T\sqcup [-e]}{\bigoplus}\mathscr V(T\sqcup [-e]-\phi[a]) \\
=\underset{j=0}{\overset{a}{\bigoplus}}\left(\underset{\mbox{\tiny $\begin{array}{c} \phi:[a]\rightarrow T\sqcup [-e] \\
|Im(\phi)\cap T|=a-j\\
\end{array}$}}{\bigoplus}\mathscr V(T\sqcup [-e]-\phi[a]) \right)
=\underset{j=0}{\overset{a}{\bigoplus}}\left(\underset{\mbox{\tiny $\begin{array}{c}\phi=(\phi',\phi'')\\ \phi':[a-j]\rightarrow T\\  \phi'':[j]\rightarrow [-e]
\end{array}$}}{\bigoplus}\mathscr V(T\sqcup [-e]-\phi[a]) \right)^{\oplus {{a}\choose{j} }}\\ =\underset{j=0}{\overset{a}{\bigoplus}}\left(\underset{\mbox{\tiny $\begin{array}{c}\phi':[a-j]\rightarrow T\\  \phi'':[j]\rightarrow [-e]
\end{array}$}}{\bigoplus}\mathscr V((T-\phi'[a-j])\sqcup ([-e]-\phi''[j])) \right)^{\oplus {{a}\choose{j} }}\\
=\underset{j=0}{\overset{a}{\bigoplus}}\left(\underset{\mbox{\tiny $\begin{array}{c}\phi':[a-j]\rightarrow T\\  
\end{array}$}}{\bigoplus}\mathscr V((T-\phi'[a-j])\sqcup [-(e-j)]) \right)^{\oplus \left({{a}\choose{j} }\cdot ([j],[e])\right)}\\
=\underset{j=0}{\overset{a}{\bigoplus}}\left(\underset{\mbox{\tiny $\begin{array}{c}\phi':[a-j]\rightarrow T\\ 
\end{array}$}}{\bigoplus}\mathbb S^{e-j}\mathscr V(T-\phi'[a-j])\right)^{\oplus \left({{a}\choose{j} }\cdot ([j],[e])\right)}
\\ =\underset{j=0}{\overset{a}{\bigoplus}}\left(\mathbb B^{-(a-j)}(\mathbb S^{e-j}\mathscr V)(T)\right)^{\oplus \left({{a}\choose{j} }\cdot ([j],[e])\right)}\\
\end{array}
\end{equation} Since $e\geq a+d$, we know that $e-j\geq d$ for each $0\leq j\leq a$. Hence, each $\mathbb S^{e-j}\mathscr V$ appearing in the direct sum 
in \eqref{eq5/20} is finitely generated. Then, each $\mathbb B^{-(a-j)}(\mathbb S^{e-j}\mathscr V)$ is finitely generated and it is now clear from \eqref{eq5/20}
that  $\mathbb S^{e}\mathbb B^{-a}(\mathscr V)$ is finitely generated. 

\end{proof}

\begin{thm}\label{P5.7}Let $Mod-\mathcal R$ be locally noetherian.  Let $\mathscr V\in FI_{\mathcal R}$ be shift finitely generated. Then, for any $a\geq  0$, $H_a(\mathscr V)$ is also shift
finitely generated.
\end{thm}

\begin{proof} From Lemma \ref{L5.6}, we know that  $\mathbb B^{-a}(\mathscr V)$ is also shift
finitely generated. We have shown in Proposition \ref{P3.11} that $FI_{\mathcal R}^{sfg}$ is a Serre subcategory. From the definitions, it is now clear that
$\widetilde{\mathbb B}^{-a}(\mathscr V)$ and hence $H_a(\mathscr V)$ lie in $FI^{sfg}_{\mathcal R}$. 

\end{proof}

\begin{Thm}\label{P5.8} Suppose that $Mod-\mathcal R$ is locally noetherian. Let $\mathscr V\in FI_{\mathcal R}$ be a shift finitely generated object. Fix $a\geq 0$ and consider  any finitely generated subobject $\mathscr W\subseteq H_a(\mathscr V)$. Then,  there exists
$N\geq 0$ such that $\mathscr W_n=0$ for all $n\geq N$. 
\end{Thm}

\begin{proof} Since $\mathscr V\in FI^{sfg}_{\mathcal R}$, we know from Proposition \ref{P5.7} that $H_a(\mathscr V)$ is shift 
 finitely generated.  We consider the trivial torsion theory $\tau_0$ on $Mod-\mathcal R$ whose torsion class is $0$. Using Proposition \ref{P3.13}, this induces a torsion class on $FI^{sfg}_{\mathcal R}$ whose torsion class $\overline{\mathcal T}^{sfg}_0$ is given by
\begin{equation}
Ob(\overline{\mathcal T}^{sfg}_0):=\{\mbox{$\mathscr V\in Ob(FI_{\mathcal R}^{sfg})$ $\vert$ Every finitely generated $\mathscr W\subseteq \mathscr V$ satisfies $\mathscr W_n=0$ for $n\gg 0$}\}
\end{equation}  Using Theorem \ref{P3.17cf}, we know that the torsion subobject of  $H_a(\mathscr V)$ is given by 
\begin{equation}\label{ystors3x}
\begin{array}{ll}
\overline{\mathcal T}_0^{sfg}(H_a(\mathscr V))(S)&=\underset{b\geq 0}{colim}\textrm{ }lim\left(
\begin{CD}H_a(\mathscr V )(S)@>\psi_b^{H_a(\mathscr V)}(S) >>{\mathbb S}^bH_a(\mathscr V )(S) @<<<\mathcal T_0( {\mathbb S}^bH_a(\mathscr V)(S))
\end{CD}\right)\\
&=\underset{b\geq 1}{colim}\textrm{ }lim\left(
\begin{CD}H_a(\mathscr V )(S)@>\psi_b^{H_a(\mathscr V)}(S) >>{\mathbb S}^bH_a(\mathscr V )(S) @<<<0
\end{CD}\right)\\
\end{array}
\end{equation} From Proposition \ref{P5.3} and the expression in \eqref{ystors3x}, it now follows that $\overline{\mathcal T}_0^{sfg}(H_a(\mathscr V))(S)=H_a(\mathscr V)(S)$. Hence,
$H_a(\mathscr V)\in \overline{\mathcal T}^{sfg}_0$ and the result follows. 

\end{proof}

We conclude with the following result. 

\begin{cor}\label{Cor6.9}  Suppose that $Mod-\mathcal R$ is locally noetherian. Let $\mathscr V\in FI_{\mathcal R}$ be a shift finitely generated object such that
$H_0(\mathscr V)$ and $H_1(\mathscr  V)$ are finitely generated. Then, there exists $N\geq 0$ such that
\begin{equation}\label{5.17xet}
\underset{\tiny \begin{array}{c} T\subseteq S \\ |T|\leq N\\ \end{array}}{colim}\textrm{ }\mathscr V(T)=\mathscr V(S)
\end{equation} for each finite set $S$.

\end{cor}

\begin{proof} Since $H_0(\mathscr V)$ and $H_1(\mathscr  V)$ are finitely generated, it follows from Theorem \ref{P5.8} that   there exists
$N\geq 0$ such that  $H_0(\mathscr V)_n=H_1(\mathscr V)_n=0$ for all $n\geq N$. The rest of the proof now follows in a manner similar to that of Theorem \ref{T5.5} :  it is clear that
\eqref{5.17xet} holds for all $S$ such that $|S|\leq N$.  We consider a set $S$ with  $|S|>N$ and suppose that \eqref{5.17xet} holds for all finite sets $U$ of cardinality $<|S|$. We observe that
\begin{equation}\label{5.18xet}
\underset{\tiny \begin{array}{c} T\subseteq S \\ |T|\leq N\\ \end{array}}{colim}\textrm{ }\mathscr V(T)=\underset{U\subsetneq S}{colim}\textrm{ }\underset{\tiny \begin{array}{c} T\subseteq U \\ |T|\leq N\\ \end{array}}{colim}\textrm{ }\mathscr V(T)
\end{equation} Since each $U$ appearing in \eqref{5.18xet} has cardinality $<|S|$, we have
\begin{equation}
\underset{\tiny \begin{array}{c} T\subseteq S \\ |T|\leq N\\ \end{array}}{colim}\textrm{ }\mathscr V(T)=\underset{U\subsetneq S}{colim}\textrm{ }\mathscr V(U)
\end{equation} Finally since $|S|>N$, we know that $H_0(\mathscr V)(S)=H_1(\mathscr V)(S)=0$. The result is now clear from the expressions in Proposition \ref{P5.2}. 

\end{proof}

\end{document}